\documentclass[12pt,a4paper,reqno]{amsart}            
\usepackage[utf8]{inputenc}

\usepackage{a4wide}
\usepackage[T1]{fontenc}
\usepackage{graphicx}
\usepackage{tikz,pgfplots}
\pgfplotsset{compat=1.17}

\title[Doubling minimizers]{Minimizing measures for the doubling condition}

\author[Benito-de la Cigoña]{Fernando Benito F. de la Cigo\~{n}a}
\address{Fernando Benito Fern\'andez de la Cigo\~{n}a\hfill\break\indent
Department of Mathematics\hfill\break\indent 151 Thayer Street\hfill\break\indent Providence, RI, 02912 USA}
\email{fbf\_delacigona@brown.edu}

\author[Cabrera Linares]{Federico Cabrera Linares}
\address{Federico Cabrera Linares \hfill\break\indent 
 Instituto de Ciencias Matem\'aticas \hfill\break\indent 
 Consejo Superior de Investigaciones Cient\'ificas \hfill\break\indent 
 C/ Nicol\'as Cabrera 13-15 \hfill\break\indent 
 28049 Madrid, Spain}
\email{01federicocabrera@gmail.com}

\author[Conde-Alonso]{Jos\'{e} M. Conde Alonso}
\address{Jos\'{e} M. Conde Alonso \hfill\break\indent 
 Departamento de Matem\'aticas \hfill\break\indent 
 Universidad Aut\'onoma de Madrid \hfill\break\indent 
 C/ Francisco Tom\'as y Valiente sn\hfill\break\indent 
 28049 Madrid, Spain}
\email{jose.conde@uam.es}

\author[Tradacete]{Pedro Tradacete}
\address{Pedro Tradacete \hfill\break\indent 
 Instituto de Ciencias Matem\'aticas \hfill\break\indent 
 Consejo Superior de Investigaciones Cient\'ificas \hfill\break\indent 
 C/ Nicol\'as Cabrera 13-15 \hfill\break\indent 
 28049 Madrid, Spain}
\email{pedro.tradacete@icmat.es}

\thanks{F. Cabrera Linares was partially funded by CSIC through the JAE Intro 2025 programme (\texttt{JAEINT\_25\_03219}). J. M. Conde-Alonso was partially supported by grants \texttt{CNS2022-135431} and \texttt{RYC2019-027910-I} (Ministerio de Ciencia, Spain). P. Tradacete was partially supported by grants \texttt{PID2020-116398GB-I00}, \texttt{PID2024-162214NB-I00} and \texttt{CEX2023-001347-S} funded by  MCIN/AEI/10.13039/501100011033. 
}

\usepackage{amsthm,amssymb,amsmath,amsopn,amsfonts,pb-diagram,accents,xcolor,array,hyperref,mathtools,enumerate,enumitem,bm,bbm,mathrsfs,amscd,esint}
\usepackage{authblk}
\usepackage{comment}
\usepackage{algpseudocode}
\usepackage{multicol}
\usepackage{mathtools}

\usepackage[subnum]{cases}

\theoremstyle{plain}
\newtheorem{theorem}{Theorem}[section]
\newtheorem{lemma}[theorem]{Lemma}
\newtheorem{lem}[theorem]{Lemma}
\newtheorem{prop}[theorem]{Proposition}
\newtheorem{cor}[theorem]{Corollary}
\theoremstyle{definition}
\newtheorem{definition}[theorem]{Definition}
\newtheorem{example}[theorem]{Example}

\numberwithin{equation}{section}
\numberwithin{figure}{section}
\numberwithin{table}{section}

\theoremstyle{plain}
\newtheorem{ltheorem}{Theorem}

\theoremstyle{definition}

\theoremstyle{remark}
\newtheorem{remark}[theorem]{Remark}

\allowdisplaybreaks
\newcommand{\N}{\mathbb{N}}
\newcommand{\Z}{\mathbb{Z}}
\newcommand{\R}{\mathbb{R}}

\newcommand{\Borel}{\mathcal B}
\newcommand{\defeq}{\coloneqq}
\newcommand{\eqdef}{\eqqcolon}
\DeclareMathOperator{\diam}{diam}
\DeclareMathOperator{\dist}{dist}

\begin{document}



\newcommand{\X}{\mathcal{X}}

\newcommand{\D}{\mathscr{D}}

\newcommand{\DM}{\mathrm{DM}}



\maketitle

\begin{abstract}
We study those measures whose doubling constant is the least possible among doubling measures on a given metric space. It is shown that such measures exist on every complete metric space supporting at least one doubling measure. In addition, a connection between minimizers for the doubling constant and superharmonic functions is exhibited. This allows us to show that for the particular case of the Euclidean space $\mathbb R^d$, Lebesgue measure is the only minimizer for the doubling constant (up to constant multiples) precisely when $d=1$ or $d=2$, while for $d\geq3$ there are infinitely many independent minimizers when considering the Euclidean metric. Analogously, in the discrete setting, we can show uniqueness of the counting measure as a minimizer for regular graphs where the standard random walk is a recurrent Markov chain. The counting measure is also shown to be a minimizer in every infinite graph where the cardinality of balls depends solely on their radii.
\end{abstract}

\section*{Introduction}

Given a metric space $(\X,d)$, a Borel measure $\mu$ on $\X$ is said to be \textit{doubling} if there exists a constant $C>0$ such that for all $x\in \X$ and $r>0$
\begin{equation*}
\mu(B(x,2r)) \leq C \mu(B(x,r))
\end{equation*}
and $0<\mu(B(x_0,r_0))<\infty$ for some $x_0\in \X$, $r_0>0$. If so, $0<\mu(B(x,r))<\infty$ for all $x\in \X$ and $r>0$.
Here and throughout the paper $B(x,r)=\{y\in \X:d(y,x)<r\}$. 

Metric spaces equipped with a doubling measure (also known in the literature as homogeneous spaces \cite{CW1971}) are central in analysis and geometry because they provide a controlled notion of ``volume growth'' that resembles Euclidean space. The doubling condition is crucial for extending classical tools, such as covering theorems, maximal functions, Poincaré inequalities and differentiation theory, to more general geometrical contexts (see the monographs \cite{BB2011, Heinonen} for these and further developments). This makes them a natural setting for analysis on fractals, manifolds, and metric spaces arising in a wide range of topics, including for instance geometric group theory or computer science.

Given a doubling measure $\mu$, let $C_\mu$ denote its \textit{doubling constant}, which is the best possible $C$ in the display above. In other words, $$C_\mu=\sup_{x,r}\frac{\mu(B(x,2r))}{\mu(B(x,r))}.$$ 
Now, if $\X$ admits a doubling measure, we denote its \emph{least doubling constant} by
$$
C_\X = \inf_{\mu} C_\mu,
$$
where the infimum ranges over all doubling measures on $\X$, and we say that $C_\X=\infty$ if $\X$ does not support any doubling measure. 
We will then call a doubling measure $\mu$ on $\X$ a \textit{doubling minimizer} if $C_\mu=C_\X$.

Although a relatively natural quantity, the research on $C_{\X}$ started only recently in \cite{ST2019}, where it was shown that $C_{\X}\geq2$ whenever $\X$ does not reduce to a singleton. The precise value of $C_\X$ is related to metric dimensions and can be relevant for instance for controlling the norm of extension operators on spaces of Lipschitz functions (see \cite{BB2006}), among other applications.

The least doubling constant of a metric space and the set of measures minimizing this constant can reveal relevant information about the space. It is easy to observe that, for $\mathbb R^d$ with the distance induced by any $p$-norm ($1\leq p\leq \infty$), one has $C_{(\mathbb R^d,\|\,\cdot\,\|_p)}=2^d$ (see \cite[Proposition 5.1]{ST2019}), and this value is achieved in particular with the Lebesgue measure. The question that motivates our research in this paper is whether the Lebesgue measure is the only minimizer, up to scalar multiples.

In the discrete setting, a connection between the least doubling constant of a graph (equipped with the shortest path distance) and the spectral radius of its (possibly infinite) adjacency matrix has been found in \cite{DST2023}. To be more precise, for a simple connected locally finite graph $G$ and a measure $\mu$ on $G$ one can consider the \textit{local doubling constant} $$C_\mu^0=\sup_{x \in G}\frac{\mu(B(x,2))}{\mu(B(x,1))},$$ and the \textit{least local doubling constant} $C_G^0=\inf_\mu C_\mu^0$. For every graph of uniformly bounded degree, $C_G^0$ coincides with $1+r(A_G)$, where $r(A_G)$ is the spectral radius of the adjacency matrix of $G$; moreover, in the case of finite graphs, where $r(A_G)$ coincides with the largest eigenvalue of $A_G$, the corresponding eigenvector (also known as Perron eigenvector) induces the unique measure minimizing $C_G^0$. In general, $C_G^0\leq C_G$ with equality under certain conditions. In particular, it can be shown that certain families of graphs, including all graphs with diameter 2, have a unique minimizer up to multiplicative constants. For the particular case of linear graphs with $n$ vertices, this is the case as long as $n\leq 8$, while for $n\geq 9$ the situation is more elusive. In the infinite linear graphs $\mathbb N$ and $\Z$ the least doubling constant is 3, and although the counting measure is the unique minimizer for $\Z$, this is not the case for $\mathbb N$ (see \cite{DST2025} for details).

In this paper, our starting point will be an argument, based on Christ's dyadic systems of cubes, that shows minimizers for the least doubling constant exist on every complete metric space which supports a doubling measure. The completeness assumption cannot be omitted in general: Example \ref{ex:noncomplete-counterexample} gives an incomplete metric space $\X$ with $C_\X=2$ but for which every doubling measure $\mu$ has $C_\mu>2$.

We also show that every value in $(C_\X,\infty)$ is the doubling constant of some measure on $\X$ whenever $\X$ has more than one point and supports a doubling measure (Proposition \ref{prop:range}). Recall that for complete metric spaces the existence of doubling measures is equivalent to the space being \textit{geometrically doubling} \cite{LS1998} (see also \cite{VK1988}), in the sense that there is a constant $N\in\mathbb N$ such that for every $r>0$, every ball of radius $2r$ can be covered by at most $N$ balls of radius $r$.

Once the question of existence of minimizers is settled, it is natural to wonder how large this set of minimizers can be and what structure it might have. For instance, are all minimizing measures mutually absolutely continuous? It is pertinent to recall here that, although doubling measures can be supported on very small sets, there are certain limitations (see \cite[Proposition 2.5]{ST2019}, \cite[I.8.6]{Stein1993}, \cite{Wu1998}). In some cases, including the canonical example of $\R^d$, a theorem of Sj\"odin \cite{S1997}, extending work of Jonsson \cite{Jonsson1995}, will allow us to show that minimizing measures are actually mutually absolutely continuous. This fact will be key in the following:

\begin{ltheorem}\label{teor:Rd}
The Lebesgue measure is the unique doubling minimizer on $\mathbb R$ and $(\mathbb R^2,\|\,{\cdot}\,\|_p)$, $1\leq p\leq\infty$, up to scalar multiples, whereas for $d\geq3$, the set of doubling minimizers on $(\mathbb R^d,\|\,{\cdot}\,\|_2)$ is an infinite-dimensional cone.
\end{ltheorem}

On $\R$, a direct discretization argument uses the uniqueness of the counting measure as a doubling minimizer on $\Z$ to establish the uniqueness of the Lebesgue measure. The corresponding uniqueness result on $\Z^2$ gives, through a similar discretization, an alternative proof on $\R^2$ for the norms $\|\,{\cdot}\,\|_1$ and $\|\,{\cdot}\,\|_\infty$. The \(\Z\) and \(\Z^2\) uniqueness results follow from the connection between doubling minimizers on regular graphs and positive discrete superharmonic functions: by P\'olya's theorem, the random walks on $\Z$ and $\Z^2$ are recurrent, so all such functions are constant.

In dimensions $d\geq2$, our proof will be based on the fact that a minimizer on $\R^d$ must be a measure with density $f$ (with respect to the Lebesgue measure) such that
$$
\frac{1}{|B(x,2r)|}\int_{B(x,2r)} f(y)dy \leq \frac{1}{|B(x,r)|}\int_{B(x,r)} f(y)dy. 
$$
For smooth densities, the symmetries of $\|\,{\cdot}\,\|_p$\kern0.1em-balls imply that the above inequality forces $\Delta f\leq0$. By Liouville's theorem, there are no nonconstant positive superharmonic functions on $\R^2$, which rules out nonconstant smooth minimizing densities for every planar $p$-norm. Since convolution with mollifiers preserves the inequality, an approximation argument gives the same conclusion without any regularity assumption. In dimensions $d\geq3$, however, positive nonconstant superharmonic functions do exist. Moreover, for the Euclidean norm, the condition $\Delta f\leq0$ is also sufficient for the inequality between averages, and we use this fact to construct an uncountable family of linearly independent minimizing densities. This sufficiency fails for other norms, so it does not settle the existence of nonconstant minimizers for $p\neq2$ in dimensions $d\geq3$. 

In the discrete setting, we also show that the counting measure $\#_{\Z^d}$ is a minimizer in $\Z^d$ in every dimension, although it may not be the only one when $d\geq 3$. This also yields a procedure to compute the explicit value of the constant $C_{\Z^d}$ in any dimension. The result is a particular case of the following theorem. 
\begin{ltheorem}\label{teor:graphs}
    Let $\X$ be a graph with \(C_\X<\infty\) such that $|B(x,r)|=|B(y,r)|$ for all $x,y\in\X$ and $r>0$. Then the counting measure is a doubling minimizer on $\X$.
\end{ltheorem}





The rest of the paper is organized as follows: Section \ref{sec:sec1} deals with general metric spaces and contains the proof of existence of doubling minimizers for complete spaces, as well as the incomplete-space counterexample, the range of doubling constants, and several other basic facts. In Section \ref{sec:sec2}, we prove Theorem \ref{teor:graphs} and we discuss examples of discrete measure spaces for which we can identify the family of doubling minimizers via its local structure. This will include, in particular, the cases of $\mathbb{Z}$ and $\mathbb{Z}^2$. Finally, Section \ref{sec:sec3} contains the proof of Theorem \ref{teor:Rd}, using the discretization arguments and developing the density and convolution methods, and concludes with the obstruction for non-Euclidean norms.

\section{General metric spaces}\label{sec:sec1}

For every metric space $(\X,d)$, the quantity $C_\X$ is defined as an infimum. We first show that, under the hypothesis of completeness, it is always attained.
\begin{theorem}\label{t:existence}
    Let $(\X,d)$ be a complete metric space with $C_\X<\infty$. Then $\X$ admits a doubling minimizer, that is, a doubling measure $\mu$ such that
    \begin{equation*}
        C_\mu=C_\X.
    \end{equation*}
\end{theorem}
For the discrete case, this was shown in \cite[Proposition 2]{DST2025}. For the more general case of complete spaces, we need to use the existence of a system of dyadic cubes on $\X$, which will allow us to discretize the space and imitate the proof for graphs. Christ constructed dyadic cubes in spaces of homogeneous type \cite{Ch1990}. Here we borrow the construction from \cite[Theorem 2.2]{HK2012}, which only requires geometric doubling. This condition holds because $\X$ supports a doubling measure. The result can be stated as follows:

\begin{theorem}[Dyadic cubes]\label{t:cubes}
    Let $(\X,d)$ be a geometrically doubling metric space. For each $k \in \Z$, there exist an (at most)  countable index set $J_k$, points $\{z_j^k\}_{j\in J_k}\subseteq\X$, and a partition $\D_k = \{Q_j^k\}_{j\in J_k}$ of $\X$ into Borel sets such that the following properties hold:
    \begin{enumerate}
        \item If $\ell \geq k$, then $Q_j^k \cap Q_{j'}^\ell \in \{\emptyset, Q_{j'}^\ell\}$.
        \item There exists a constant $C_0\geq1$ such that, for all $k$ and $j$, \[B(z_j^k, 2^{-k}) \subseteq Q_j^k \subseteq B(z_j^k, 2^{-k}C_0).\]
    \end{enumerate}
\end{theorem}

As usual, we denote
\[
    \D=\bigcup_{k\in\Z}\D_k,\qquad
    \D^+=\bigcup_{k\geq0}\D_k.
\]
As a consequence of (2) in Theorem \ref{t:cubes}, every open set is a
countable union of cubes from \(\D^+\). Since the cubes
are Borel sets, it follows that \(\D^+\) generates the Borel
\(\sigma\)-algebra of \(\X\), that is,
\[
    \sigma(\D^+)=\Borel(\X).
\]

However, prescribing values on dyadic cubes does not by itself guarantee
the existence of a Borel measure with those values. To construct our
limiting measure, we will use Carath\'eodory's extension theorem on the
ring of sets
\[
    \mathcal R=\left\{\bigsqcup_{r=1}^N Q_r:
        N\geq1,\ Q_r\in\D^+\right\}\cup\{\emptyset\}.
\]

We will need the following version of Christ's small-boundary estimate
\cite{Ch1990}, with constants uniform over measures whose
doubling constants have a common bound.

\begin{lemma}\label{lem:small-boundaries}
Let \(C_*\geq1\). There are constants \(A,\eta>0\), depending
only on \(C_*\) and on the dyadic system, such that every doubling measure
\(\nu\) with \(C_\nu\leq C_*\) satisfies
\[
    \nu\bigl(\{x\in Q_j^k:\dist(x,\X\setminus Q_j^k)<t2^{-k}\}\bigr)
    \leq A t^\eta \nu(Q_j^k)
\]
for all \(k\in\mathbb Z\), \(j\in J_k\), and \(0<t\leq1\).
\end{lemma}

\begin{proof}
We fix a cube \(Q=Q_j^k\) and put
\[
    \mathcal W_\ell
    =\{Q_i^\ell\in\D_\ell:Q_i^\ell\subseteq Q,\ \dist(Q_i^\ell,\X\setminus Q)<2^{-\ell}\},
    \qquad W_\ell=\bigcup_{Q_i^\ell\in\mathcal W_\ell}Q_i^\ell,\quad \ell\geq k.
\]
By nesting, we have \(W_{\ell'}\subseteq W_\ell\subseteq Q\) whenever
\(\ell'\geq\ell\). Also, \(W_\ell\) contains every \(x\in Q\) with
\(\dist(x,\X\setminus Q)<2^{-\ell}\).

We choose \(m\in\N\) such that \(2^m-2C_0\geq1\). We fix
\(Q_i^\ell\in\mathcal W_\ell\) and let \(Q_s^{\ell+m}\subseteq Q_i^\ell\)
be the cube containing \(z_i^\ell\). Since
\(B(z_i^\ell,2^{-\ell})\subseteq Q_i^\ell\subseteq Q\) and
\(\diam(Q_s^{\ell+m})\leq2C_0\,2^{-\ell-m}\), we have
\[
    \dist(Q_s^{\ell+m},\X\setminus Q)
    \geq2^{-\ell}-2C_0\,2^{-\ell-m}
    \geq2^{-\ell-m}.
\]
Thus \(Q_s^{\ell+m}\notin\mathcal W_{\ell+m}\), and consequently
\(Q_s^{\ell+m}\cap W_{\ell+m}=\emptyset\) and
\begin{equation}\label{eq:descendant-disjoint}
    W_{\ell+m}\cap Q_i^\ell\subseteq Q_i^\ell\setminus Q_s^{\ell+m}.
\end{equation}

We now compare the measures of \(Q_i^\ell\) and \(Q_s^{\ell+m}\). We have
\(Q_i^\ell\subseteq B(z_s^{\ell+m},2C_0\,2^{-\ell})\) and
\(B(z_s^{\ell+m},2^{-\ell-m})\subseteq Q_s^{\ell+m}\). Taking
\(L=\lceil\log_2(2C_0\,2^m)\rceil\) and iterating the doubling condition,
we obtain
\begin{align*}
    \nu(Q_i^\ell)
    &\leq \nu\bigl(B(z_s^{\ell+m},2C_0\,2^{-\ell})\bigr)
    \leq \nu\bigl(B(z_s^{\ell+m},2^L\,2^{-\ell-m})\bigr)\\
    &\leq C_\nu^L\,\nu\bigl(B(z_s^{\ell+m},2^{-\ell-m})\bigr)
    \leq C_*^L\,\nu(Q_s^{\ell+m}).
\end{align*}
Combining this with \eqref{eq:descendant-disjoint} yields, for
\(\delta=(2C_*^L)^{-1}\in(0,1)\),
\[
    \nu(W_{\ell+m}\cap Q_i^\ell)
    \leq\nu(Q_i^\ell)-\nu(Q_s^{\ell+m})\leq(1-\delta)\nu(Q_i^\ell).
\]
Summing over all \(Q_i^\ell\in\mathcal W_\ell\), we get
\(\nu(W_{\ell+m})\leq(1-\delta)\nu(W_\ell)\). Iteration yields
\[
    \nu(W_{k+qm})\leq(1-\delta)^q\nu(W_k)\leq(1-\delta)^q\nu(Q),\qquad q\geq0.
\]

Finally, for \(0<t\leq1\), we take
\(q=\lfloor\log_2(1/t)/m\rfloor\). Since
\(t2^{-k}\leq2^{-k-qm}\), it follows that
\[
    \nu\bigl(\{x\in Q:\dist(x,\X\setminus Q)<t2^{-k}\}\bigr)
    \leq\nu(W_{k+qm})
    \leq(1-\delta)^q\nu(Q)
    \leq\frac{t^\eta}{1-\delta}\nu(Q),
\]
where \(\eta=-\log_2(1-\delta)/m>0\). This proves the estimate with
\(A=(1-\delta)^{-1}\), and both constants depend only on \(C_*\) and
\(C_0\).
\end{proof}

\begin{prop}\label{prop:dyadic-extension}
Let \(\X\) be complete and geometrically doubling, and consider a sequence
\(\{\mu_n\}_{n\in\N}\) of doubling measures on \(\X\) with
\(\sup_n C_{\mu_n}<\infty\). Assume that, for every \(Q\in\D^+\), the limit
\[
    a(Q)=\lim_{n\to\infty}\mu_n(Q)
\]
exists. Then the formula
\[
    a\left(\bigsqcup_{r=1}^N Q_r\right)=\sum_{r=1}^N a(Q_r),
    \qquad Q_r\in\D^+,
\]
defines a \(\sigma\)-finite premeasure on \(\mathcal R\), which extends uniquely to a Borel measure \(\mu\) on \(\X\) satisfying
\(\mu(Q)=a(Q)\) for every \(Q\in\D^+\).
\end{prop}

\begin{proof}
For \(E=\bigsqcup_{r=1}^N Q_r\in\mathcal R\), we have
\[
    \sum_{r=1}^N a(Q_r)
    =\lim_{n\to\infty}\sum_{r=1}^N\mu_n(Q_r)
    =\lim_{n\to\infty}\mu_n(E).
\]
Consequently, \(a(E)\) is well-defined and, setting
\(a(\emptyset)=0\), we obtain a finite, finitely additive function on
\(\mathcal R\).

We first show that, for every \(E\in\mathcal R\) and \(\varepsilon>0\),
there is an open set \(U\) such that
\begin{equation}\label{eq:uniform-boundary-neighborhood}
    \partial E\subseteq U,\qquad \sup_n\mu_n(U)<\varepsilon.
\end{equation}
We write \(E\) as a union of cubes of a common generation \(h\geq0\).
Each \(Q_i^h\) is either contained in \(E\) or disjoint from \(E\), so
\(\partial E\subseteq\overline{\X\setminus Q_i^h}\) in both cases.
For \(0<t\leq1\), we put
\[
    U_t=\{x\in\X:\dist(x,\partial E)<t2^{-h}\}.
\]
These sets are open and contain \(\partial E\). Since, for all $x\in\X$,
\[
    \dist(x,\X\setminus Q_i^h)
    =\dist\bigl(x,\overline{\X\setminus Q_i^h}\bigr)
    \leq\dist(x,\partial E),
\]
we have
\[
    U_t\cap Q_i^h
    \subseteq\{x\in Q_i^h:\dist(x,\X\setminus Q_i^h)<t2^{-h}\}.
\]
The set \(U_1\) is bounded, so only finitely many cubes of generation \(h\)
meet it. Applying Lemma \ref{lem:small-boundaries} with
\(C_*=\sup_n C_{\mu_n}\), we get
\vspace{0.15cm}
\begin{align*}
    \sup_n\mu_n(U_t)
    =\sup_n
    \sum_{\mathmakebox[2.5em]{Q_i^h\cap \,U_1\neq\emptyset}}
    \mu_n(U_t\cap Q_i^h)
    \leq A t^\eta\,\sup_n
    \sum_{\mathmakebox[2.5em]{Q_i^h\cap \,U_1\neq\emptyset}}
    \mu_n(Q_i^h)
    \leq A t^\eta
    \sum_{\mathmakebox[2.5em]{Q_i^h\cap \,U_1\neq\emptyset}}
    \sup_n\mu_n(Q_i^h).
\end{align*}
The last expression tends to zero as \(t\downarrow0\), since the sum is
finite and independent of \(t\). Choosing
\(t\) sufficiently small proves \eqref{eq:uniform-boundary-neighborhood}.

We now prove continuity from above at the empty set. Let
\(E_1\supseteq E_2\supseteq\cdots\) be sets in \(\mathcal R\) with
\(\bigcap_m E_m=\emptyset\). The set \(\overline{E_1}\) is compact:
it is bounded, hence totally bounded by geometric doubling, and complete
as a closed subset of \(\X\). The compact set
\(K=\bigcap_{m\geq1}\overline{E_m}\) satisfies
\[
    K\subseteq\bigcup_{m\geq1}\partial E_m,
\]
since every point of \(K\) fails to belong to some \(E_m\), while
belonging to its closure.

Given \(\varepsilon>0\), we choose open sets \(U_m\supseteq\partial E_m\)
with \(\sup_n\mu_n(U_m)<\varepsilon2^{-m}\), using
\eqref{eq:uniform-boundary-neighborhood}. Their union
\(U=\bigcup_{m\geq1}U_m\) contains \(K\) and satisfies
\[
    \sup_n\mu_n(U)
    \leq\sup_n\sum_{m\geq1}\mu_n(U_m)
    \leq\sum_{m\geq1}\sup_n\mu_n(U_m)<\varepsilon.
\]
The compact sets \(\overline{E_m}\setminus U\) decrease to the empty set,
so \(\overline{E_M}\setminus U=\emptyset\) for some \(M\). Therefore, for \(m\geq M\), we have \(E_m \subseteq E_M \subseteq \overline{E_M}\subseteq U\), and thus
\[
    a(E_m)=\lim_{n\to\infty}\mu_n(E_m)
    \leq\sup_n\mu_n(U)<\varepsilon.
\]
Hence \(a(E_m)\to0\), and the continuity criterion for finite, finitely
additive functions shows that \(a\) is a premeasure \cite[Theorem 3.2]{Bauer2001}.

Finally, \(a\) is \(\sigma\)-finite because
\(\X=\bigcup_{j\in J_0}Q_j^0\) and each \(a(Q_j^0)\) is finite.
By Carath\'eodory's extension theorem \cite[Theorem 5.6]{Bauer2001},
\(a\) extends uniquely to a measure on \(\sigma(\mathcal R)=\Borel(\X)\).
\end{proof}

Now we are ready to prove Theorem \ref{t:existence}.
\begin{proof}[Proof of Theorem \ref{t:existence}]
Since \(\X\) supports a doubling measure, it is geometrically doubling.
We fix a dyadic system as in Theorem \ref{t:cubes} and a cube
\(Q_1^0\in\D_0\). Let \(\{\widetilde\mu_n\}_{n\in\N}\) be a sequence of
doubling measures on \(\X\) such that \(C_{\widetilde\mu_n}\to C_\X\).
We normalize them by defining
\[
    \mu_n\defeq \frac{\widetilde\mu_n}{\widetilde\mu_n(Q_1^0)},
\]
so that \(\mu_n(Q_1^0)=1\) for all \(n\in\N\). Rescaling does not modify
the doubling constant, so \(C_{\mu_n}\to C_\X\) as well.

Let us now see that \(\{\mu_n(Q)\}\) is bounded for every \(Q\in\D^+\).
For each \(j\in J_0\), choose \(R\geq1\) such that
\(Q_j^0\subseteq B(z_1^0,R)\). Since \(B(z_1^0,1)\subseteq Q_1^0\),
iterating the doubling condition gives
\[
    \mu_n(Q_j^0)
    \leq\mu_n(B(z_1^0,R))
    \leq C_{\mu_n}^{\lceil\log_2 R\rceil}\mu_n(B(z_1^0,1))
    \leq\bigl(\sup_n C_{\mu_n}\bigr)^{\lceil\log_2 R\rceil}.
\]
The sequence \(\{C_{\mu_n}\}\) is convergent and hence bounded. Thus the
masses of cubes of generation \(0\) are bounded uniformly in \(n\).
Every \(Q\in\D^+\) is contained in one of these cubes, which proves the
claim.

Since \(\D^+\) is countable, Cantor diagonalization yields a subsequence,
which we will still denote by \(\{\mu_n\}\), such that \(\lim_{n\to\infty}\mu_n(Q)\)
exists for every \(Q\in\D^+\). By Proposition
\ref{prop:dyadic-extension}, there is a Borel measure \(\mu\) such that
\[
    \mu(Q)=\lim_{n\to\infty}\mu_n(Q),\qquad Q\in\D^+.
\]
It is finite on balls, since each ball is contained in a
finite union of cubes of generation \(0\).

We next show that \(\mu_n(B(x,r))\to\mu(B(x,r))\) whenever
\(\mu(\{y\in\X:d(x,y)=r\})=0\). Fix such a ball and let
\(\varepsilon>0\). By continuity of \(\mu\) from below and above, we
can choose \(0<r_-<r<r_+\) such that
\[
    \mu(B(x,r)\setminus B(x,r_-))<\varepsilon,\qquad
    \mu(B(x,r_+)\setminus B(x,r))<\varepsilon.
\]
We take \(h\geq0\) sufficiently large that
\(2C_0\,2^{-h}<\min\{r-r_-,r_+-r\}\), and let \(R_-\) and \(R_+\)
be the unions of cubes in \(\D_h\) meeting \(B(x,r_-)\) and \(B(x,r)\),
respectively. These are finite unions, so \(R_-, R_+\in\mathcal R\), and
the bound on the diameters of the cubes gives
\[
    B(x,r_-)\subseteq R_-\subseteq B(x,r)
    \subseteq R_+\subseteq B(x,r_+).
\]
Using convergence on \(R_-\) and \(R_+\), we obtain
\begin{align*}
    \mu(B(x,r))-\varepsilon
    &\leq\mu(R_-)
    \leq\liminf_{n\to\infty}\mu_n(B(x,r))\\
    &\leq\limsup_{n\to\infty}\mu_n(B(x,r))
    \leq\mu(R_+)\leq\mu(B(x,r))+\varepsilon.
\end{align*}
Since \(\varepsilon\) is arbitrary, the desired convergence follows.

We finally claim that \(C_\mu=C_\X\). Fix \(x\in\X\) and \(r>0\).
Since \(\mu\) is finite on balls and spheres with center \(x\) are
pairwise disjoint, only countably many of them can have positive
\(\mu\)-measure. We may therefore choose \(0<s_m\uparrow r\), with
\(s_m<r\), such that the spheres of radii \(s_m\) and \(2s_m\) both
have \(\mu\)-measure zero. By the convergence just proved,
\[
    \mu(B(x,2s_m))
    =\lim_{n\to\infty}\mu_n(B(x,2s_m))
    \leq \lim_{n\to\infty}C_{\mu_n}\mu_n(B(x,s_m))
    =C_\X\mu(B(x,s_m)).
\]
Letting \(m\to\infty\) and using continuity from below gives
\[
    \mu(B(x,2r))\leq C_\X\mu(B(x,r)).
\]
Since \(\mu(Q_1^0)=1\), every ball has positive \(\mu\)-measure:
otherwise, iterating this inequality would give measure zero to a ball
containing \(Q_1^0\). Thus \(\mu\) is a doubling measure with
\(C_\mu\leq C_\X\), and the reverse inequality follows from the
definition of \(C_\X\).
\end{proof}
The completeness of \(\X\) was used to construct the limiting measure.
The following probabilistic example shows that this assumption cannot be omitted from
Theorem \ref{t:existence}: an incomplete metric space may admit doubling
measures without admitting a minimizer.

\begin{example}\label{ex:noncomplete-counterexample}
    Let \(Y=\{0,1\}^{\N}\). For \(x\neq y\), define
\[
    d(x,y)=2^{-m},\qquad \text{where } m=\min\{j\geq1:x_j\neq y_j\},
\]
and set \(d(x,x)=0\). This is the usual ultrametric on the boundary of the infinite
binary tree. For a finite word \(u=(u_1,\ldots,u_n)\in \{0,1\}^{n}\), we define the cylinder
\[
    [u]=\{x\in Y:x_1=u_1,\ldots,x_n=u_n\}.
\]
The open balls in \(Y\) are precisely these cylinders. The space \(Y\) is compact:
it is complete and, for each \(n\geq1\), the \(2^n\) cylinders of length \(n\)
cover \(Y\) and are open balls of radius \(2^{-n}\), so \(Y\) is totally bounded.

Consider the subspace
\[
    \X=\left\{x\in Y:
    \limsup_{N\to\infty}\frac1N\sum_{j=1}^N x_j<\frac12\right\},
\]
which is a Borel subset. \(\X\) is not all of \(Y\) and it is dense in \(Y\), because every cylinder contains
the sequence that continues with zeros. Since \(Y\) is complete,
density and properness imply that \(\X\) is not complete.

For an open ball \(B_\X(x,r)\) in \(\X\), if \(n\in\N\) is such that \(2^{-(n+1)}<r\leq2^{-n}\), then
\(B_\X(x,r)=\X\cap[x|n]\), where \(x|n\defeq(x_1,\ldots,x_n)\).
Doubling the radius gives \(B_\X(x,2r)=\X\cap[x|n-1]\).
Thus, the doubling quotients for a measure \(\mu\) on \(\X\) would be
\[
    \frac{\mu(\X\cap[u])}{\mu(\X\cap[u0])},\qquad
    \frac{\mu(\X\cap[u])}{\mu(\X\cap[u1])}.
\]

For any \(q\in(0,1)\), let \(\nu_q\) be the Bernoulli product probability on
\(Y\) determined by
\[
    \nu_q([u])=q^{\#_1(u)}(1-q)^{\#_0(u)},
\]
where \(\#_1(u)\) and \(\#_0(u)\) count the ones and zeros in \(u\), respectively. The
coordinate maps are independent Bernoulli random variables with mean \(q\).
By the strong law of large numbers \cite[Theorem 2.4.1]{Durrett2019}, if
\(q<1/2\), then \(\nu_q(\X)=1\). We may therefore regard \(\nu_q\) as a
probability measure on \(\X\). Since \(\nu_q(\X\cap[u])=\nu_q([u])\), we obtain
\[
    \frac{\nu_q(\X\cap[u])}{\nu_q(\X\cap[u1])}=\frac1q,\qquad
    \frac{\nu_q(\X\cap[u])}{\nu_q(\X\cap[u0])}=\frac1{1-q},
\]
for every finite word \(u\).
Thus \(C_{\nu_q}=1/q\) for \(q<1/2\). Letting \(q\uparrow1/2\) gives
\(C_\X\leq2\). By the universal lower bound \(C_\X\geq2\) \cite[Theorem 3.1]{ST2019}, we deduce
\[
    C_\X=2.
\]

Assume, towards a contradiction, that there is a doubling measure \(\mu\) on
\(\X\) with \(C_\mu=2\). For every word \(u\),
\[
    \frac{\mu(\X\cap[u])}{\mu(\X\cap[u0])}\leq2,\qquad
    \frac{\mu(\X\cap[u])}{\mu(\X\cap[u1])}\leq2.
\]
Since \(\X\cap[u]=(\X\cap[u0])\sqcup(\X\cap[u1])\), both inequalities can hold only if
\[
    \mu(\X\cap[u0])=\mu(\X\cap[u1])=\frac{\mu(\X\cap[u])}2.
\]
By induction, we get
\[
    \mu(\X\cap[u])=2^{-|u|}\mu(\X),
\]
for every finite word \(u\), where \(|u|\) denotes its length.

We extend \(\mu\) to \(Y\) by \(\widetilde\mu(A)=\mu(A\cap \X)\). Then
\[
    \widetilde\mu([u])=2^{-|u|}\mu(\X)
\]
for every cylinder. By uniqueness in Carath\'eodory's extension theorem,
\(\widetilde\mu=\mu(\X)\nu_{1/2}\), where \(\nu_{1/2}\) is the symmetric
Bernoulli probability on \(Y\). Applying the strong law again, now with \(q=1/2\),
we obtain \(\nu_{1/2}(\X)=0\). Therefore
\[
    \mu(\X)=\widetilde\mu(\X)=\mu(\X)\nu_{1/2}(\X)=0,
\]
a contradiction. Hence no minimizer exists on \(\X\).
\end{example}

An argument in \cite{ST2019} shows that, unless $\X$ is a singleton, one can always construct measures with an arbitrarily large doubling constant by perturbing a given doubling measure. Hence, no maximizers for the doubling condition exist. What remains to be understood is which intermediate values can occur as doubling constants. We show that, in fact, every value larger than $C_\X$ is attained by a doubling measure.
\begin{prop}\label{prop:range}
    Let \(\X\) be a metric space with more than one point and \(C_\X<\infty\).
Then
\[
    (C_\X,\infty)\subseteq
    \{C_\mu:\mu \text{ doubling measure on } \X\}.
\]
\end{prop}
\begin{proof}
    Fix \(\Lambda>C_\X\) and choose a doubling measure \(\nu\) with
\(
    C_\nu<\Lambda.
\)
Taking distinct points \(x_0,y_0\in \X\) and setting
\[
    r_0=\frac{3}{4}d(x_0,y_0),\qquad B_0=B(x_0,r_0),
\]
we note that
\(
    B(y_0,d(x_0,y_0)/4)\subseteq B(x_0,2r_0)\setminus B(x_0,r_0),
\)
and thus
\[
    C_0\defeq\frac{\nu(B(x_0,2r_0))}{\nu(B(x_0,r_0))}>1.
\]

For each \(t\in[0,1)\), we define a measure \(\mu_t\) by setting
\[
    \mu_t(F)=\nu(F)-t\nu(F\cap B_0),
\]
for every Borel set \(F\subseteq\X\).
Each \(\mu_t\) is nontrivial and moreover doubling: for any ball \(B\),
\[
    \frac{\mu_t(2B)}{\mu_t(B)} \leq \frac{\nu(2B)}{(1-t)\nu(B)} \leq \frac{C_\nu}{1-t},
\]
where \(2B\) denotes the ball with the same center and twice the radius.

Now we define the function \(g \colon [0,1) \to [C_\mathcal{X}, \infty)\) given by
\[
    g(t)=C_{\mu_t},\qquad t\in[0,1).
\]
On the one hand, \(g(0)=C_\nu<\Lambda\). On the other hand, evaluating at \(B_0\) gives
\[
    g(t)\geq
    \frac{\mu_t(2B_0)}{\mu_t(B_0)}
    =
    \frac{\nu(2B_0) - t\nu(B_0)}{(1-t)\nu(B_0)}
    =
    \frac{C_0-t}{1-t}\longrightarrow\infty
    \qquad (t\uparrow1).
\]

Let us check that \(g\) is continuous. Fix an arbitrary \(\delta\in(0,1)\). We show that \(g\) is in fact Lipschitz on \([0,1-\delta]\). For
each ball \(B\), introduce the function \(R_B \colon [0, 1) \to [1, \infty)\),
\[
    R_B(t)=\frac{\mu_t(2B)}{\mu_t(B)}, \qquad t\in[0,1).
\]
We estimate the derivative of \(R_B\) for a fixed ball \(B\): for an arbitrary \(t\in[0,1-\delta]\),
\begin{align*}
    |R_B'(t)|&=\left|\left(\frac{\nu(2B)-t\nu(2B\cap B_0)}{\nu(B)-t\nu(B\cap B_0)}\right)'\right|
    = \frac{|\nu(B\cap B_0)\nu(2B)-\nu(2B\cap B_0)\nu(B)|}{(\nu(B)-t\nu(B\cap B_0))^2}\\[0.25cm]
    &\leq\frac{2\nu(B)\nu(2B)}{\delta^2\nu(B)^2}
    =\frac{2\nu(2B)}{\delta^2\nu(B)}\leq\frac{2C_\nu}{\delta^2}.
\end{align*}
Thus, the functions \(R_B\) are uniformly Lipschitz on \([0,1-\delta]\).
The supremum of a uniformly Lipschitz family is Lipschitz with the same
constant, so \(g(t)=\sup_B R_B(t)\) is also Lipschitz on \([0,1-\delta]\), and therefore continuous on
\([0,1)\). We now apply the intermediate value theorem to obtain
\(t\in(0,1)\) such that \(C_{\mu_t}=g(t)=\Lambda\).
\end{proof}

Hence, depending on whether the infimum \(C_\X\) is attained or not, the set of doubling constants
\(\{C_\mu:\mu\text{ doubling on }\X\}\) is either \([C_\X,\infty)\) or \((C_\X,\infty)\).
For example, when \(\X\) is complete, this set is \([C_\X,\infty)\), whereas for the space in
Example \ref{ex:noncomplete-counterexample}, it is \((C_\X,\infty)=(2,\infty)\).

Now we turn our attention to the structure of the family of doubling minimizers.
Given a metric space $(\X,d)$, let us denote this set by
$$
\DM(\X,d)=\{\mu:C_\mu=C_\X\}.
$$
When there is no ambiguity regarding the metric, we will simply write $\DM(\X)$. 
\begin{prop} Let $(\X,d)$ be a metric space.
\begin{enumerate}
    \item[(i)] $\DM(\X,d)$ is a convex cone.
    \item[(ii)] $\DM(\X,d)$ is invariant under isometries.
\end{enumerate}
\end{prop}

\begin{proof}
(i) Clearly the doubling constant is invariant under scalar multiplication, $C_{\alpha\mu}=C_\mu$ for every $\alpha>0$. Take now $\mu_1,\mu_2\in \DM(\X,d)$ and let $\mu=\mu_1+\mu_2$. For $x\in \X$ and $r>0$, we have
\begin{align*}
\frac{\mu(B(x,2r))}{\mu(B(x,r))}&=\frac{\mu_1(B(x,2r))+\mu_2(B(x,2r))}{\mu_1(B(x,r))+\mu_2(B(x,r))}\\
&\leq \max\Big\{\frac{\mu_1(B(x,2r))}{\mu_1(B(x,r))},\frac{\mu_2(B(x,2r))}{\mu_2(B(x,r))}\Big\}\\
&\leq \max\{C_{\mu_1},C_{\mu_2}\}=C_{\X}.
\end{align*}
This shows that $\mu\in \DM(\X,d)$ as claimed. 

(ii) Let $\phi$ denote an isometry on $(\X,d)$ (i.e., a bijective map $\phi:\X\rightarrow \X$ such that $d(x,y)=d(\phi x, \phi y)$ for every $x,y\in \X$). Given $\mu\in \DM(\X,d)$, let $\mu_\phi(A)=\mu(\phi(A))$ for every Borel set $A\subseteq \X$. Clearly, we have
\[
C_{\mu_\phi}=\sup_{x,r}\frac{\mu(\phi(B(x,2r)))}{\mu(\phi(B(x,r)))}=\sup_{x,r}\frac{\mu(B(\phi(x),2r))}{\mu(B(\phi(x),r))}=C_{\mu}.\qedhere
\]
\end{proof}

This argument actually works for the family of all measures with doubling constant bounded by any fixed $C>0$, showing that it is a convex cone invariant under isometries.

As is customary, for a locally integrable function $f$ defined on a measure space $(\Omega, \Sigma, \mu)$, and $A\in\Sigma$, we denote the corresponding average as
$$
\fint_A fd\mu\defeq \frac{1}{\mu(A)}\int_A fd\mu.
$$
Given a metric space $(\X,d)$ and $\mu$ a Borel measure on $(\X,d)$, we will say a function $f:\X\rightarrow \mathbb R$ is \textit{superharmonic with respect to $\mu$} if for every $x\in \X$ and $r>0$ 
\begin{equation}\label{eq:superharmonic}
        \fint_{B(x,2r)}fd\mu\leq\fint_{B(x,r)}fd\mu.
\end{equation}
Our use of the term superharmonicity is specific to the study of the doubling condition and differs from that in \cite[Definition 3.3]{AGG2019}, where one requires
\[
    f(x)\geq\fint_{B(x,r)}f\,d\mu.
\]
Both definitions extend the notion of \textit{(strongly) harmonic functions} on a metric measure space $(\X,d,\mu)$, which are defined by the property 
$$
f(x)= \fint_{B(x,r)}f d\mu,
$$
for every $x\in \X$ and every $r>0$. These were introduced and analyzed in \cite{AGG2019, GG2009}.
The connection between our notion of superharmonicity and classical superharmonicity in Euclidean spaces will be discussed in Section \ref{sec:sec3}.

The following fact (whose straightforward proof is omitted) shows that a wealth of superharmonic functions allows us to construct new minimizers from known ones.

\begin{prop}
    If $\mu\in \DM(\X,d)$ and $\nu$ is a Borel measure, finite and nonzero on balls and absolutely continuous with respect to $\mu$, whose Radon--Nikodym derivative is superharmonic with respect to $\mu$, then $\nu\in \DM(\X,d)$.
\end{prop}

A natural question here is whether all doubling minimizers can be described by a density with respect to a given minimizer. Note that it is well known that there exist mutually singular doubling measures on $[0,1]$ \cite[Theorem 3]{AB1956}. In fact, there are examples of doubling measures in $\mathbb R^d$ that assign positive measure to a rectifiable curve \cite{GKS2010} (although all of them must vanish on hypersurfaces). We will show next the absolute continuity of minimizers under very specific conditions, related to the Hausdorff measure of the space and its minimal doubling constant. These follow from the work of Sj\"odin \cite{S1997} which in turn generalizes the work of Jonsson \cite{Jonsson1995} in $\mathbb R^d$. Let us first recall the following.

\begin{definition}\label{defd-medida}
A Borel measure $\mu$ on a metric space $\X$ is called a \textit{$d$-measure} ($d>0$) if there exist a constant $C\geq1$ and $r_0>0$ such that for all $x\in \X$, $0<r\leq r_0$
\begin{equation}\label{ineqregular}
\frac{1}{C}r^d\leq\mu(B(x,r))\leq Cr^d.
\end{equation}
\end{definition}

This is slightly weaker than the well-known notion of Ahlfors $d$-regular measure. In particular, note that a metric space which supports a $d$-measure must have Hausdorff dimension equal to $d$ \cite[Proposition 1.1]{S1997}. We will also need the following.

\begin{definition}
Let $\mu$ be a Borel measure on a metric space $\X$, positive and finite on every ball, and let $C\geq 1$, $d\geq0$. We say that $\mu$ is \textit{$(C,d)$-homogeneous} if it satisfies
\begin{equation}\label{cshomogenea}
\mu(B(x,\lambda r))\leq C\lambda^d\mu(B(x,r))
\end{equation}
for all $x\in \X$, $r>0$, and $\lambda\geq 1$.
\end{definition}

A $(C,d)$-homogeneous measure $\mu$ is, in particular, doubling with $C_\mu\leq C2^d$. Conversely, a doubling measure is $(C_\mu, \log_2 C_\mu)$-homogeneous. From this notion, one can define the \textit{Volberg–Konyagin dimension} (see \cite{LS1998}) as
\begin{equation*}
\dim_{VK}(\X)\defeq \inf \{d:\textrm{there exists a }(C,d)\textrm{-homogeneous measure }\mu\textrm{ on }\X\}.
\end{equation*}
In particular, we have $\dim_{VK}(\X)\leq \log_2 C_\X$. Volberg-Konyagin dimension can be thought of as a measure theoretic counterpart of the well-known notion of Assouad's dimension of a metric space. In fact, it is easy to see that the Hausdorff dimension of a metric space is always bounded above by $\dim_{VK}(\X)$. Therefore, on a metric space of Hausdorff dimension $d$ we have that $C_\X\geq 2^d$. The following relevant observation, which applies for instance to the particular case of $\mathbb R^d$, follows from \cite{S1997}.

\begin{prop}\label{minimizersabsolcont}
    \enlargethispage{2\baselineskip}
    If $\X$ is a complete metric space with a $d$-measure $m$ such that $C_\X=2^d$, then $m$ is mutually absolutely continuous with every doubling minimizer on $\X$.
\end{prop}

\begin{proof}
    Let $\mu\in \DM(\X)$, so that $C_\mu=2^d$. As noted above, the doubling measure $\mu$ is $(C_\mu,\log_2 C_\mu)$-homogeneous, hence $(2^d,d)$-homogeneous. By \cite[Theorem 1]{S1997}, $\mu$ and $m$ are mutually absolutely continuous.
\end{proof}

\section{Discrete measure spaces} \label{sec:sec2}

We turn our attention to discrete measure spaces. We use standard terminology for graphs. We assume $\X$ is the set of vertices of a simple, connected, locally finite graph with edge set $E(\X)$. Given $x\in \X$, its neighbors form the set
$$
N(x) = \{ y \in \X: (x,y) \in E(\X)\},
$$
and we call $|N(x)|$ the \textit{degree} of $x$. On $\X$ we always consider the shortest path distance, so that $N(x) = \{y \in \X: d(x,y) = 1\}$. Measures $\mu$ on $\X$ are in one-to-one correspondence with densities $f_\mu:\X\to \mathbb{R}_+$ with respect to the counting measure $\#_\X$. We shall pay special attention to the local behavior of measures $\mu$. To that end, define
$$
C_\mu^0 = \sup_{x\in \X} \frac{\mu(B(x,2))}{\mu(B(x,1))}=\sup_{x\in \X} \frac{\mu(N(x)\cup\{x\})}{\mu(\{x\})}, 
$$
and say that the \textit{local case is critical} for $(\X,\mu)$ if $C_\mu=C_\mu^0$. The hypothesis of the local case being critical will allow us to obtain several positive results in the sequel. The \textit{discrete (combinatorial) Laplacian} of a function $f:\X \to \mathbb{R}$ is defined by
\begin{equation*}
        \Delta f(x)\defeq \frac{1}{|N(x)|}\sum_{y\in N(x)} (f(x)-f(y))=f(x)-\frac{1}{|N(x)|}\sum_{y\in N(x)}f(y).
\end{equation*}
We say $f$ is \textit{harmonic} if $\Delta f(x)=0$ for all $x\in \X$. If $\Delta f(x)\geq 0$ for all $x\in \X$ we call it \textit{superharmonic}. In the graph setting, the notion of (super)harmonicity is a local one, and hence it can differ significantly from the generalization of the one used in the Introduction for $\R^d$. Let us now assume that $\X$ is \textit{regular}, that is, there exists $N_0$ such that $|N(x)|=N_0$ for all $x\in \X$. If the local case is critical for $(\X,\#_\X)$, any doubling minimizer $\mu$ must satisfy $C_\mu^0\leq C_{\#_\X}^0$. But this means that
\begin{equation*}
    1+\frac{\sum_{y\in N(x)}\mu(y)}{\mu(x)} \leq \sup_{x\in\X}\frac{\mu(x)+\sum_{y\in N(x)}\mu(y)}{\mu(x)}=C_\mu^0 \leq C_{\#_\X}^0=1+|N(x)|,
\end{equation*}
which in turn implies
\begin{equation*}
    \frac{\sum_{y\in N(x)}\mu(y)}{|N(x)|}\leq \mu(x) \mbox{ for all } x \in \X. 
\end{equation*}
 The above discussion implies that, when the local case is critical, if $\mu$ is a doubling minimizer, then $f_\mu$ must be a positive, superharmonic function. We next use the well-known link between the existence of nontrivial positive superharmonic functions and transience of random walks to find obstructions to the existence of doubling minimizers in certain spaces. In particular, we will use the following fact: if $\X$ is a graph where the standard random walk is a recurrent Markov chain, then the only positive superharmonic functions on $\X$ are constants \cite[Theorem 1.16]{W2000}. By Theorem 1 in \cite{L1983}, this always happens when $\X$ is a regular graph such that there exists $c>0$ so that 
\begin{equation}\label{eq:CondRecurrencia}
|B(x,n)|\leq cn^2. 
\end{equation}

\begin{lem}\label{grafosGeneral}
    Let $\X$ be a regular graph such that the local case is critical for $(\X,\#_\X)$. If the standard random walk on $\X$ is recurrent, then $C_\X=\deg(\X)+1$ and
    \[
        \DM(\X)=\{\alpha\cdot\#_\X:\,\alpha>0\}.
    \]
\end{lem}
\begin{proof}
     Since $\X$ is regular, then
     \begin{equation*}
         \frac{|B(y,2)|}{|B(y,1)|}=\deg(\X)+1
     \end{equation*}
     for all $y\in \X$, and therefore the local case being critical implies $C_{\#_\X}=\deg(\X)+1$. Since $\X$ is complete, Theorem \ref{t:existence} implies \(\DM(\X) \neq \emptyset\). For any such minimizer $\mu$, we must have
     \begin{equation*}
         \frac{\mu(B(y,2))}{\mu(B(y,1))}\leq C_\X\leq \deg(\X)+1
     \end{equation*}
     for all $y\in \X$, and so $f_\mu:\X\to\R^+$ must be a superharmonic function. But the random walk on $\X$ is recurrent, so all positive superharmonic functions must be constant.
\end{proof}

\begin{remark}
    In \cite{DST2023}, it is shown that $C_\X\geq r(A_\X)+1$, where $r(A_\X)$ is the spectral radius of the adjacency matrix of $\X$, and the equality holds for certain finite graphs. In fact, \cite[Proposition 18]{DST2023} provides a criterion in terms of the Perron eigenvector of the adjacency matrix of a finite graph $\X$ that characterizes when $C_\X= r(A_\X)+1$. Lemma \ref{grafosGeneral} can be seen as a reinforcement of the conclusions there in terms of the properties of the random walk in $\X$.
\end{remark}
Although without retrieving uniqueness, Theorem \ref{teor:graphs} allows us to conclude the counting measure is a doubling minimizer regardless of whether the local case is critical or not. We now provide its proof.
\begin{proof}[Proof of Theorem \ref{teor:graphs}]
     Suppose for a contradiction there is a measure $\mu$ on $\X$ such that $C_{\mu}<C_{\#_\X}$. Take a sequence of radii $n_k$ such that
    \begin{equation*}
        \frac{\#_\X(B(x,2n_k))}{\#_\X(B(x,n_k))}\xrightarrow{k\to\infty}C_{\#_\X}.
    \end{equation*}
    Fix some $k$ sufficiently large so that $C_\mu<\#_\X(B(x,2n_k))/\#_\X(B(x,n_k))$ and set $n=n_k$. Then, for some $\gamma<1$,
    \begin{equation*}
        \frac{\mu(B(x,2n))}{\#_\X(B(x,2n))}\leq\gamma\frac{\mu(B(x,n))}{\#_\X(B(x,n))}
    \end{equation*}
    for all $x\in\X$. Fix a point $0 \in\X$. The rest of the argument relies on the following observation: fixed $r>0$, for every $N\geq r$ and $x\in B(0,N-r)$ we have
    \begin{equation*}
        \lbrace y\in B(0,N):x\in B(y,r)\rbrace=B(x,r).
    \end{equation*}
    Furthermore, for every $x\in\X$ we have
    \begin{equation*}
        \lbrace y\in B(0,N):x\in B(y,r)\rbrace \subseteq B(x,r).
    \end{equation*}
    Therefore, for $N\geq2n$ we get
    \begin{align*}
        \mu(B(0,N-2n))&=\sum_{x\in B(0,N-2n)}\frac{\#_\X(B(x,2n))}{\#_\X(B(x,2n))}\mu(\lbrace x\rbrace)\\
        &=\sum_{x\in B(0,N-2n)}\sum_{\substack{y\in B(0,N):\\x\in B(y,2n)}}\frac{\mu(\lbrace x\rbrace)}{\#_\X(B(x,2n))}\\
        &= \sum_{y\in B(0,N)}\sum_{\substack{x\in B(0,N-2n) \cap B(y,2n)}}\frac{\mu(\lbrace x\rbrace)}{\#_\X(B(x,2n))} \\
        & \leq\sum_{y\in B(0,N)}\sum_{x\in B(y,2n)}\frac{\mu(\{x\})}{\#_\X(B(x,2n))} = \sum_{y\in B(0,N)}\frac{\mu(B(y,2n))}{\#_\X(B(y,2n))}\\
        &\leq \gamma \sum_{y\in B(0,N)}\frac{\mu(B(y,n))}{\#_\X(B(y,n))} = \frac{\gamma}{\#_\X(B(0,n))}\sum_{y\in B(0,N)}\sum_{x\in B(y,n)}\mu(\lbrace x\rbrace)\\&\leq\gamma\mu(B(0,N+n)).
    \end{align*}
    In particular, taking $N=2M-n$ above, we get $\mu(B(0,2M))\geq\gamma^{-1}\mu(B(0,2M-3n))$ if $2M\geq 3n$. Iterating this yields:
    \begin{equation*}
        \mu(B(0,2M))\geq \frac{1}{\gamma}\mu(B(0,2M-3n))\geq \left(\frac{1}{\gamma}\right)^k\mu(B(0,2M-3nk))
    \end{equation*}
    if $k\leq\frac{2M}{3n}$. Taking $k=\lfloor M/3n\rfloor$, choosing $M$ large enough yields a contradiction with $\mu$ being doubling as
    \begin{equation*}
        \mu(B(0,2M))\geq\left(\frac{1}{\gamma}\right)^{\lfloor M/3n\rfloor}\mu(B(0,M)).\qedhere
    \end{equation*}
\end{proof}




\subsection{\texorpdfstring{The case of $\mathbb{Z}^d$}{The case of Z to the d}} For the remainder of this subsection, we denote by $B_d(x,k)$ the $\Z^d$-ball with center $x$ and radius $k\in\N$. We focus on the counting measure, which only requires computing the size of $B_d(0,n)$, where $0=(0,\ldots,0)\in\Z^d$, by translation invariance. It is easy to see that $|B_1(0,k)|=2k-1$ for each $k\in\N$. This allows one to compute the size of the balls in dimension 2.  
\begin{lem}\label{bolas2}
For all $k\in \N$, $|B_2(0,k)|=2k^2-2k+1=k^2+(k-1)^2$.
\end{lem}
\begin{proof}
Define
\begin{equation*}
	H_0\defeq B_2(0,k)\cap\{(y_1,y_2):y_1=0\}=\{(0,y_2): |y_2|<k\}.
\end{equation*}
We have $|H_0|=|B_1(0,k)|$. Similarly, if 
\begin{equation*}
	H_1^+\defeq B_2(0,k)\cap\{(y_1,y_2):y_1=1\}=\{(1,y_2): |y_2|<k-1\},
\end{equation*}
then $|H_1^+|=|B_1(0,k-1)|$, and the same happens to
\begin{equation*}
	H_1^-=B_2(0,k)\cap\{(y_1,y_2):y_1=-1\}.
\end{equation*}
Analogously, given $1\leq j\leq k-1$ we may set
\begin{equation*}
    H_j^\pm\defeq B_2(0,k)\cap\{(y_1,y_2):y_1=\pm j\}=\{(\pm j,y_2): |y_2|<k-j\},
\end{equation*}
and $|H_j^+|=|H_j^-|=|B_1(0,k-j)|$. It is easy to see that
\begin{equation*}
	B_2(0,k)=\bigsqcup_{j=1}^{k-1}(H_j^+\sqcup H_j^-)\sqcup H_0,
\end{equation*}
which yields
\begin{align*}
	|B_2(0,k)|& =|B_1(0,k)|+2\sum_{j=1}^{k-1}|B_1(0,k-j)| \\
    & =2k-1+2\sum_{j=1}^{k-1}(2j-1)=2k^2-2k+1.\qedhere
\end{align*}
\end{proof}

\begin{remark}
The technique used to prove Lemma \ref{bolas2} gives us an inductive algorithm to compute $P_d(k)\defeq |B_d(0,k)|$ for each $d$ and $k$. Indeed, if $P_d(k)$ is known, we may set
\begin{equation*}
	H_0\defeq B_{d+1}(0,k)\cap\{(y_1,\ldots, y_{d+1}):y_1=0\},
\end{equation*}
and $|H_0|=|B_d(0,k)|$. Similarly, for $0\leq j \leq k-1$
\begin{equation*}
	|H^\pm_j|=|\{(y_1,\ldots, y_{d+1})\in B_{d+1}(0,k):y_1=\pm j\}|=|B_d(0,k-j)|,
\end{equation*}
and we still have the decomposition
\begin{equation*}
	B_{d+1}(0,k)=\bigsqcup_{j=1}^{k-1}(H_j^+\sqcup H_j^-)\sqcup H_0.
\end{equation*}
Therefore,
\begin{equation*}
	|B_{d+1}(0,k)|=|B_d(0,k)|+2\sum_{j=1}^{k-1}|B_d(0,k-j)|=P_d(k)+2\sum_{j=1}^{k-1}P_d(k-j).
\end{equation*}
Assuming $P_d(k)$ is a polynomial of degree $d$, the sum in the formula above  gives a polynomial of degree $d+1$ via Faulhaber's formula. We next list what one gets in dimensions $3$ and $4$:
\begin{align*}
	P_3(k)&=\frac{4}{3}k^3-2k^2+\frac{8}{3}k-1, \\
	P_4(k)&=\frac{2}{3}k^4-\frac{4}{3}k^3+\frac{10}{3}k^2-\frac{8}{3}k+1.
\end{align*}    
\end{remark}
We can now try to compute
\begin{equation*}
	C_{\Z^d}=C_{\#_d}=\sup_{k\in\N} \frac{P_d(2k)}{P_d(k)}.
\end{equation*}
At least in the case $1\leq d\leq 4$ the supremum above is attained. Indeed, for each $d$ the function $P_d(2r)/P_d(r)$ is eventually decreasing, so one only has to check a finite number of values to calculate its maximum. Figure \ref{fig-cocientes} depicts the functions $P_d(2r)/P_d(r)$ for $1\leq d\leq4$ and their maxima.

\begin{figure}[ht]
\centering
\resizebox{0.62\textwidth}{!}{%
\begin{tikzpicture}
    \begin{axis}[
        width=9.4cm,
        height=8.2cm,
        xmin=0,
        xmax=11.85,
        ymin=0,
        ymax=19.9,
        axis lines=left,
        axis line style={black, line width=0.45pt},
        tick style={black, line width=0.35pt},
        xlabel={radius \(r\)},
        ylabel={\(P_d(2r)/P_d(r)\)},
        xlabel style={font=\small},
        ylabel style={font=\small},
        xtick={0,1,2,3,4,5,6,7,8,9,10,11},
        ytick={0,2,4,6,8,10,12,14,16,18},
        grid=major,
        major grid style={black!12, line width=0.35pt},
        clip=false,
        enlargelimits=false,
        tick label style={font=\scriptsize, black},
        legend style={
            at={(0.8025,0.655)},
            anchor=center,
            draw=black!25,
            fill=white,
            fill opacity=0.9,
            text opacity=1,
            font=\scriptsize,
            cells={anchor=west}
        },
        reverse legend,
        legend image post style={mark=*},
        every axis plot/.append style={line width=1.0pt, mark=*, mark size=1.35pt}
    ]
        \addplot[
            green!55!black,
            mark options={fill=green!55!black, draw=green!30!black, line width=0.25pt}
        ] coordinates {
            (1,3)
            (2,2.333333333333333)
            (3,2.2)
            (4,2.142857142857143)
            (5,2.111111111111111)
            (6,2.090909090909091)
            (7,2.076923076923077)
            (8,2.066666666666667)
            (9,2.058823529411765)
            (10,2.052631578947368)
            (11,2.047619047619048)
        };
        \addlegendentry{\(d=1\)}

        \addplot[
            blue!75!black,
            mark options={fill=blue!75!black, draw=blue!30!black, line width=0.25pt}
        ] coordinates {
            (1,5)
            (2,5)
            (3,4.6923076923076925)
            (4,4.52)
            (5,4.414634146341464)
            (6,4.344262295081967)
            (7,4.294117647058823)
            (8,4.256637168141593)
            (9,4.227586206896552)
            (10,4.204419889502762)
            (11,4.18552036199095)
        };
        \addlegendentry{\(d=2\)}

        \addplot[
            red!80!black,
            mark options={fill=red!80!black, draw=red!35!black, line width=0.25pt}
        ] coordinates {
            (1,7)
            (2,9)
            (3,9.24)
            (4,9.126984126984127)
            (5,8.984496124031008)
            (6,8.861471861471861)
            (7,8.761273209549071)
            (8,8.68)
            (9,8.613445378151261)
            (10,8.558239861950) 
            (11,8.511851377322229)
        };
        \addlegendentry{\(d=3\)}

        \addplot[
            gray!70!black,
            mark options={fill=gray!72, draw=gray!45!black, line width=0.25pt}
        ] coordinates {
            (1,9)
            (2,14.333333333333334)
            (3,16.609756097560975)
            (4,17.372093023255815)
            (5,17.573208722741434)
            (6,17.57562408223201)
            (7,17.508921644685803)
            (8,17.42124051762606)
            (9,17.331049602630858)
            (10,17.245346569756428)
            (11,17.1664872622892)
        };
        \addlegendentry{\(d=4\)}

        \addplot[
            only marks,
            mark=*,
            mark size=2.8pt,
            green!55!black,
            mark options={fill=green!55!black, draw=black, line width=0.45pt},
            forget plot
        ] coordinates {(1,3)};
        \addplot[
            only marks,
            mark=*,
            mark size=2.8pt,
            blue!75!black,
            mark options={fill=blue!75!black, draw=black, line width=0.45pt},
            forget plot
        ] coordinates {(1,5) (2,5)};
        \addplot[
            only marks,
            mark=*,
            mark size=2.8pt,
            red!80!black,
            mark options={fill=red!80!black, draw=black, line width=0.45pt},
            forget plot
        ] coordinates {(3,9.24)};
        \addplot[
            only marks,
            mark=*,
            mark size=2.8pt,
            gray!70!black,
            mark options={fill=gray!72, draw=black, line width=0.45pt},
            forget plot
        ] coordinates {(6,17.57562408223201)};

        \node[
            green!55!black,
            font=\scriptsize,
            anchor=south,
            fill=white,
            inner sep=1pt
        ] at (axis cs:1,3.22) {\((1,3)\)};
        \node[
            blue!75!black,
            font=\scriptsize,
            anchor=south,
            fill=white,
            inner sep=1pt
        ] at (axis cs:1.5,5.28) {\((1,5)\ (2,5)\)};
        \node[
            red!80!black,
            font=\scriptsize,
            anchor=south,
            fill=white,
            inner sep=1pt
        ] at (axis cs:3,9.55) {\((3,231/25)\)};
        \node[
            gray!70!black,
            font=\scriptsize,
            anchor=south,
            fill=white,
            inner sep=1pt
        ] at (axis cs:6,17.82) {\((6,11969/681)\)};
    \end{axis}
\end{tikzpicture}%
}
\caption{Quotients $P_d(2r)/P_d(r)$ for $1\leq d\leq4$. Highlighted points mark the maxima over positive integer radii.}\label{fig-cocientes}
\end{figure}

It can be checked that $C_{\Z}=3$, $C_{\Z^2}=5$, $C_{\Z^3}=231/25=9.24$ and $C_{\Z^4}=11969/681\approx 17.58$.  In any case, $C_{\Z^d}$ has to grow exponentially: since $P_d(k)$ is a polynomial of degree $d$, we must have
    $$
    C_{\Z^d}\geq \lim_{k\to\infty} \frac{P_d(2k)}{P_d(k)} = \lim_{k\to\infty} \frac{(2k)^d}{k^d} = 2^d. 
    $$
One important consequence of the above computations about the size of the balls in $\mathbb{Z}^d$ is that both $C_{\#_1}$ and $C_{\#_2}$ are attained at small balls. This allows us to finish the discussion characterizing all doubling minimizers in those two cases.

\begin{cor}\label{cor:minsZZ2}
    $$
    \DM(\Z) = \{\alpha \cdot \#_\Z: \,\alpha >0\} \quad \mbox{ and } \quad \DM(\Z^2) = \{\alpha \cdot \#_{\Z^2}: \,\alpha >0\}.
    $$
\end{cor}

\begin{proof}
    By the computations above, the local case is critical for both $\Z$ and $\Z^2$, while the random walk is recurrent in $\Z$ and $\Z^2$, so there are no nonconstant positive superharmonic functions. Therefore, in both cases the only minimizer is the counting measure, or a constant multiple of it.
\end{proof}

\begin{remark}
The case of $\X=\mathbb{Z}$ had already been studied in \cite{DST2025}. The following elementary argument shows that the counting measure is the only doubling minimizer: if $\mu$ is a measure on $\mathbb{Z}$ with $C_\mu\leq 3$, suppose there is $j\in\Z$ such that $\varepsilon\defeq\mu(\{j-1\})-\mu(\{j\})>0$. Since $B(j,1)=\{j\}$ and $B(j,2)=\{j-1,j,j+1\}$, we have
    \begin{equation*}
        3\mu(\{j\})\geq \mu(\{j-1\})+\mu(\{j\})+\mu(\{j+1\})=2\mu(\{j\})+\varepsilon+\mu(\{j+1\}),
    \end{equation*}
    because $C_\mu\leq3$. Therefore, $\mu(\{j+1\})\leq \mu(\{j\})-\varepsilon$. But looking now at the balls $B(j+1,1)$ and $B(j+1,2)$, it follows that $\mu(\{j+2\})\leq \mu(\{j+1\})-\varepsilon\leq \mu(\{j\})-2\varepsilon$. Iterating this process we get
    \begin{equation*}
        \mu(\{j+n\})\leq \mu(\{j\})-n\varepsilon.
    \end{equation*}
    Therefore, for $n>\mu(\{j\})\varepsilon^{-1}$ we find $\mu(\{j+n\})<0$, and so $\mu$ is not non-negative. This argument is simpler than the one offered above or that in \cite{DST2025}, but it does not carry over to $\mathbb{N}$, where other minimizing examples can be found.
\end{remark}

We end this subsection with a natural conjecture:

\begin{remark}
    As explained above, in dimensions $1$ and $2$ the local case is critical, but in higher dimensions that is not always the case. Additionally, if $d \geq 3$, then the natural random walk is transient, and so there are nonconstant positive superharmonic functions in $\mathbb{Z}^d$. That could be interpreted as evidence that there should be more minimizers other than the counting measure. One such function for $d=3$ seems to be $f(x)=\|x\|_2^{-1/2}$ ($x\neq0$), $f(0)=1$. If we set $\mu=fd\#_{\Z^3} $, it is difficult to estimate $C_\mu$, but some preliminary computational work shows that
    \begin{equation*}
        \frac{\mu(B_3(x,2n))}{\mu(B_3(x,n))}<9.24=C_{\#_3},
    \end{equation*}
    for $x$ and $n$ not too large. The maximum seems to be attained at $n=3$ and approaches $9.24$ as $\|x\|$ grows, which makes us think that $\mu$ could be a doubling minimizer. In analogy with our results in the continuous setting, we conjecture that there are plenty of independent doubling minimizers in $\Z^d$ for $d\geq 3$. Whether this is the case is, to the best of our knowledge, open at the moment. 
\end{remark}

\subsection{Other graphs} We end this section illustrating the general observations above with a few examples of regular graphs for which we can compute all doubling minimizers.

\begin{example}
    Let $T$ be the graph whose vertices are those of a plane tessellation by regular triangles, and with edges corresponding to the sides of these triangles. $T$ is a regular, planar graph of degree $6$. One can see that, for $j\geq 2$, 
    \begin{equation*}
        |B_T(x,j)\setminus B_T(x,j-1)|=6(j-1), \mbox{ and so } |B_T(x,k)|=1+\sum_{j=1}^{k-1}6j=3k^2-3k+1.
    \end{equation*}
    The function $f(k) = |B_T(x,2k)|/|B_T(x,k)| $ attains its maximum over the natural numbers at $k=1$, so the local case is critical for $T$. By \eqref{eq:CondRecurrencia}, the random walk on $T$ is recurrent, and so by Lemma \ref{grafosGeneral}, we have $C_T = \deg(T)+1=7$ and  
    \begin{equation*}
        \DM(T)=\lbrace \alpha\cdot \#_T : \, \alpha >0\rbrace.
    \end{equation*}
    It is interesting to compare the behavior of this regular graph of degree $6$ with $\Z^3$, where the local case is not critical.
\end{example}

\begin{remark}
    If in the above example we remove a few vertices, the behavior changes significantly. Consider \(H\) the tessellation of the plane by regular hexagons, which is a regular graph where each vertex has degree $3$. One can see that
     \begin{equation*}
        |B_H(x,j)\setminus B_H(x,j-1)|=3(j-1). 
    \end{equation*} 
    for each $j>1$ and $x\in H$. Then 
    \begin{equation*}
        |B_H(x,k)|=1+\sum_{j=1}^{k-1}3j=\frac{3}{2}k^2-\frac{3}{2}k+1
    \end{equation*}
    for all $k$, and the function $f(k)=|B_H(x,2k)|/|B_H(x,k)|$ attains its global maximum at $k=2$, with $f(2)=19/4$, i.e., the local case is not critical for $H$. Therefore, there may exist doubling minimizers on $H$ different from $\#_H$. In any case, $\#_H\in\mathrm{DM}(H)$ by Theorem \ref{teor:graphs} and $C_H=19/4$.  
\end{remark}

\begin{example}
Consider the $3$-regular graph $E$ with vertices
\begin{equation*}
    \lbrace k_a: k\in\Z,1\leq a\leq 6\rbrace,
\end{equation*}
and edges
    \begin{align*}
        \left\{ \right. & \lbrace k_1,k_2\rbrace,\lbrace k_1,k_3\rbrace,\lbrace k_2,k_3\rbrace,\lbrace k_3,k_4\rbrace,\lbrace k_4,k_5\rbrace,\lbrace k_5,k_2\rbrace, \\ & \left. \lbrace k_4,k_6\rbrace,\lbrace k_5,k_6\rbrace,\lbrace k_6,(k+1)_1\rbrace:k\in\Z \right\}. 
    \end{align*}
Figure \ref{fig-grafoC} shows a portion of $E$. The rate of growth of the size of balls is linear in the radius, and therefore, by \eqref{eq:CondRecurrencia} the random walk on $E$ is recurrent. On the other hand, one can check that the local case is critical for $E$ (although the cardinality of balls of a fixed radius depends on their centers), and so by Lemma \ref{grafosGeneral}, $C_E=4$ and
    \begin{equation*}
        \DM(E)=\lbrace\alpha \cdot \#_E:\, \alpha > 0\rbrace.
    \end{equation*}

    \begin{figure}[ht]
        \centering
        \IfFileExists{grafoC.png}{\includegraphics[scale=0.4]{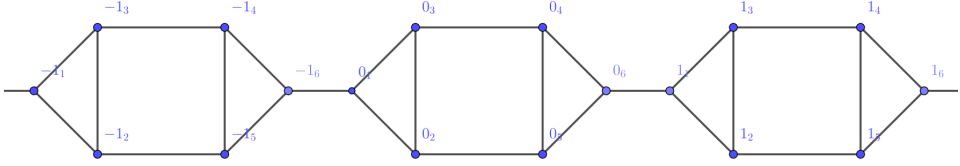}}{%
        \begin{tikzpicture}[scale=.87,every node/.style={font=\scriptsize}]
        \foreach \k in {-1,0,1} {
            \begin{scope}[xshift=4.5*\k cm]
            \coordinate (v\k1) at (0,0);
            \coordinate (v\k2) at (1,.8);
            \coordinate (v\k3) at (1,-.8);
            \coordinate (v\k4) at (2.4,-.8);
            \coordinate (v\k5) at (2.4,.8);
            \coordinate (v\k6) at (3.4,0);
            \foreach \a/\b in {1/2,1/3,2/3,3/4,4/5,5/2,4/6,5/6}
                \draw (v\k\a)--(v\k\b);
            \foreach \a in {1,...,6} {
                \fill (v\k\a) circle (1.8pt);
            }
            \foreach \a in {2,5}
                \node[above=3pt] at (v\k\a) {$\k_{\a}$};
            \foreach \a in {3,4}
                \node[below=3pt] at (v\k\a) {$\k_{\a}$};
            \foreach \a in {1,6}
                \node[above=4pt] at (v\k\a) {$\k_{\a}$};
            \end{scope}
        }
        \draw (v-16)--(v01) (v06)--(v11);
        \draw[dashed] (-5.3,0)--(v-11) (v16)--(8.7,0);
        \end{tikzpicture}}
        \caption{A piece of $E$.}\label{fig-grafoC}
    \end{figure}

\end{example}

\begin{example}
For $m\in\N$, consider the graph $K_m\times\Z$ with vertices
 \begin{equation*}
        \lbrace k_a: k\in\Z, 1\leq a\leq m\rbrace,
\end{equation*}
and edges
\begin{equation*}
    \lbrace \lbrace k_a,k_b\rbrace:k\in\Z, 1\leq a,b\leq m, a\neq b\rbrace\cup\lbrace \lbrace k_a,(k+1)_a\rbrace,\lbrace k_a,(k-1)_a\rbrace:k\in\Z, 1\leq a\leq m\rbrace. 
\end{equation*}
$K_m\times\Z$ generalizes the integers --and indeed $K_1\times\Z=\Z$. It is easy to see that $|B(x,2)|=m+2$ and $|B(x,k)|-|B(x,k-1)|=2m$ if $k \geq 3$. Therefore, $|B(x,k)|=m(2k-3)+2$ for $k\geq 2$. As above, the local case is critical and the balls grow sub-quadratically, so 
\begin{equation*}
    \DM(K_m\times\Z)=\lbrace \alpha\cdot \#_{K_m\times\Z}: \, \alpha >0\rbrace,
\end{equation*}
for all $m\in\N$.
\end{example}

\section{\texorpdfstring{Results for $\mathbb{R}^d$}{Results for R to the d}} \label{sec:sec3}
We now address the question of what the doubling minimizers in $\R^d$ are. 
A simple geometric argument in \cite[Proposition 5.1]{ST2019} shows that
\[
    C_{(\R^d,\|\,\cdot\,\|_p)}=2^d,\qquad 1\leq p\leq\infty,
\]
and it follows that the Lebesgue measure is a doubling minimizer on $(\R^d,\|\,{\cdot}\,\|_p)$. Our problem is then reduced to studying whether there are any other measures minimizing the doubling constant. We are able to answer this question in dimension one and in dimension two with any $p$-norm by proving uniqueness, whereas in dimensions three and above we exhibit an uncountable family of minimizers, but only for the Euclidean norm. 

Proposition \ref{minimizersabsolcont} allows us to reduce the problem to the study of densities with respect to the Lebesgue measure, an approach we will use in dimensions two and above. In $\R$, however, a direct discretization argument using the uniqueness result already obtained for $\Z$ suffices.
We write $\lambda$ for the Lebesgue measure in the ambient dimension. 

\begin{prop}\label{prop:Rminimizers}
    \[\DM(\R)=\lbrace\alpha\lambda : \alpha >0\rbrace.\]
\end{prop}
\begin{proof}
    Let $\mu$ be a doubling measure on $(\R,|\cdot|)$ such that $C_\mu=2$. Since $\R$ has no isolated points, $\mu$ has no atoms \cite[Proposition 2.5]{ST2019}, so endpoints may be omitted in the interval identities below. We claim that 
\begin{equation}\label{eq:triplante}
\mu(B(x,3r))\leq3\mu(B(x,r))\mbox{ for all }x\in \R, r>0.   
\end{equation}
Indeed, 
\begin{align*}
        4\mu(B(x,r))&\geq2\mu(B(x,2r))=2(\mu(B(x+r,r))+\mu(B(x-r,r)))\\&\geq\mu(B(x+r,2r))+\mu(B(x-r,2r))\\&=\mu(B(x,r))+\mu(B(x,3r)).
\end{align*}
Next, for $k,j\in\Z$, let $I_j^k=[2^{-k}j,2^{-k}(j+1))$ denote the standard dyadic intervals. For each $k\in\Z$ we consider the natural discretization $\mu_k$ of $\mu$ in $\Z$ given by
\begin{equation}\label{discretization}
    \mu_k(\{j\})=\mu(I_j^k).
\end{equation}
Let us see that $C_{\mu_k}=3$ for all $k\in\Z$. Indeed, take $j\in\Z$ and $r>0$. Let $n\in\N$ be such that $n-1<r\leq n$. By \eqref{eq:triplante} we have
\begin{align*}
    \mu_k(B(j,2r))&\leq \mu\bigl((2^{-k}(j-2n+1),2^{-k}(j+2n))\bigr)=\mu\left(B\left(2^{-k}\left(j+\frac{1}{2}\right), \frac{(4n-1)2^{-k}}{2}\right)\right)\\
    &\leq3\mu\left(B\left(2^{-k}\left(j+\frac{1}{2}\right),\frac{2^{-k}(4n-1)}{6}\right)\right)\\
    &=3\mu\left(\left(2^{-k}\left(j-\frac{2n}{3}+\frac{2}{3}\right),2^{-k}\left(j+\frac{2n}{3}+\frac{1}{3}\right)\right)\right) \\
    & \leq 3\mu\bigl((2^{-k}(j-n+1),2^{-k}(j+n))\bigr) \leq3\mu_k(B(j,r)).
\end{align*}
By Corollary \ref{cor:minsZZ2}, $\mu_k$ is a multiple of $\#_{\Z}$ for each $k\in\Z$.
Let $\alpha=\mu([0,1))$. Then, for $k\leq0$ it follows that
    \begin{equation*}
        \mu(I_j^k)=\mu_k(j)=\mu_k(0)=\mu(I_0^k)=\mu\Bigg(\bigcup_{i=0}^{2^{-k}-1}I_i^0\Bigg)=\sum_{i=0}^{2^{-k}-1}\mu_0(i)=2^{-k}\alpha.
    \end{equation*}
    Similarly, if $k>0$:
    \begin{equation*}
        \alpha=\mu\Bigg(\bigcup_{i=0}^{2^k-1}I_i^k\Bigg)=\sum_{i=0}^{2^k-1}\mu(I_i^k)=2^k\mu(I_j^k).
    \end{equation*}
    Hence, $\mu(I_j^k)=\alpha2^{-k}=\alpha \lambda(I_j^k)$ for all $j,k\in\Z$. Since the dyadic intervals generate the Borel $\sigma$-algebra, the uniqueness part of Carath\'eodory's extension theorem gives $\mu=\alpha\lambda$. 
\end{proof}

We also have uniqueness of the counting measure as a minimizer in $\Z^2$, so it is reasonable to expect that the same argument as in Proposition \ref{prop:Rminimizers} yields uniqueness of the Lebesgue measure as a minimizer in $\R^2$. However, one should keep in mind that in $\R^2$ there are different (equivalent) norms. The norms $\|\,{\cdot}\,\|_1$ and $\|\,{\cdot}\,\|_\infty$ are the ones that admit these graph-norm like discretizations. Hence, for these two norms, one can infer the uniqueness of the Lebesgue measure as a minimizer from that of the counting measure in the discrete case. Although these cases will also follow from the general argument below, we retain this proof to make the connection with the discrete case explicit.

\begin{prop}\label{prop:discretR2}
    \[\DM(\R^2,\|\,{\cdot}\,\|_1)=\DM(\R^2,\|\,{\cdot}\,\|_\infty)=\lbrace\alpha\lambda : \alpha > 0\rbrace.\]
\end{prop}
\begin{proof}
    Let $\mu$ be a doubling measure in $(\R^2,\|\,{\cdot}\,\|_1)$ such that $C_\mu=4$. For each $k\in\Z$ and $(x_0,y_0)\in \Z^2$ we define the dyadic $\|\,{\cdot}\,\|_1$\kern0.1em-ball
    \begin{equation*}
        Q_{(x_0,y_0)}^k=\{(x,y)\in\R^2: |x-2^{-k}x_0|+|y-2^{-k}y_0|<2^{-k}\}.
    \end{equation*}
    For each $k\in\Z$, we consider the discretization $\mu_k$ of $\mu$ in $\Z^2$ given by
    \begin{equation*}
        \mu_k(\{(x,y)\})=\mu(Q_{(x,y)}^k).
    \end{equation*}
     We will show that these discretizations are multiples of the counting measure. Recall that a doubling measure in $\R^d$ charges no measure on hypersurfaces (see, for instance, \cite[I.8.6]{Stein1993}). In particular, $\mu\left(\partial Q_{(x_0,y_0)}^k\right)=0$, and so it makes no difference to consider open or half-open versions of these balls. For $k\in\Z$ and $(x,y)\in\Z^2$ it holds that
     \begin{align*}
         \mu_k(\{(x,y)\})+\mu_k(\{(x+1,y)\})&+\mu_k(\{(x-1,y)\})+\mu_k(\{(x,y+1)\})+\mu_k(\{(x,y-1)\})\\
         &=\mu(B((2^{-k}x,2^{-k}y),2^{-k+1}))+\mu(B((2^{-k}x,2^{-k}y),2^{-k}))\\ 
         &\leq5\mu(B((2^{-k}x,2^{-k}y),2^{-k}))=5\mu_k(\{(x,y)\}).
     \end{align*}
     Hence, the function $(x,y)\mapsto\mu_k(\{(x,y)\})$ is positive and superharmonic on $\Z^2$, and therefore constant. Thus, for each $k\in\Z$, $\mu_k$ is a multiple of the counting measure. Set $\alpha=\mu_0(\{(0,0)\})$. Observe that, for $k\leq 0$, there are $2^{-2k}$ disjoint sets of the type $Q^0_{(x,y)}$ covering $Q^k_{(0,0)}$ (up to sets of $\mu$-measure $0$). Then, as $\mu_0$ assigns the same measure to all points, we have
    \begin{equation*}
        \mu_k(\{(0,0)\})=\alpha2^{-2k}=\frac{\alpha}{2}\lambda\left(Q^k_{(0,0)}\right).
    \end{equation*}
    Similarly, for $k>0$, there are $2^{2k}$ sets of the type $Q^k_{(x,y)}$ covering $Q^0_{(0,0)}$. Considering $\mu_k$ assigns the same measure to all points, we also have
    \begin{equation*}
        \mu_k(\{(0,0)\})=\alpha2^{-2k}=\frac{\alpha}{2}\lambda\left(Q^k_{(0,0)}\right).
    \end{equation*}
    As the dyadic $\|\,{\cdot}\,\|_1$\kern0.1em-balls generate the Borel $\sigma$-algebra, uniqueness of the Borel extension gives $\mu=(\alpha/2)\lambda$. Since $(\R^2, \|\,{\cdot}\,\|_\infty)$ is isometric to $(\R^2, \|\,{\cdot}\,\|_1)$ by a rotation and dilation, the same conclusion holds in that case.
\end{proof}

To study the set of minimizers in $(\R^d,\|\,{\cdot}\,\|_p)$ for general $1\leq p\leq\infty$, the discretization argument above does not apply directly. However, the connection with positive superharmonic functions will prove to be fruitful again. First recall that, by Proposition \ref{minimizersabsolcont}, all minimizers in these spaces are absolutely continuous with respect to each other and, in particular, with respect to the Lebesgue measure. The problem is then reduced to studying which positive measurable functions $u$ yield a doubling minimizer of the form $d\mu=u\,d\lambda$.

The fact that the Lebesgue measure reaches the least doubling constant $2^d$ in every ball allows us to characterize these densities through an inequality between averages. Indeed, given a positive locally integrable function $u$, we easily obtain that $d\mu=u\,d\lambda\in\DM(\R^d,{\|\,{\cdot}\,\|_p})$ if and only if
\begin{equation}\label{eq:decreasingAv-original}
    \fint_{B_p(x,2r)}u\,d\lambda\leq\fint_{B_p(x,r)}u\,d\lambda
\end{equation}
for every $x\in\R^d$ and $r>0$, where $B_p(x,r)$ denotes the ball in $(\R^d,\|\,{\cdot}\,\|_p)$. 

As we shall see, the only such densities when $d=2$ are constant almost everywhere, for every $1\leq p\leq\infty$, whereas there is a large family of nonconstant examples for the Euclidean norm when $d\geq3$. Our arguments are based on the connection with classical superharmonicity for smooth functions. However, some care is needed, as we cannot assume any regularity of the densities \textit{a priori}. We will overcome this difficulty by regularizing them through convolution.

We shall call a density satisfying \eqref{eq:decreasingAv-original} \textit{$p$-superharmonic in the mean sense}. It is precisely \eqref{eq:superharmonic} for the Lebesgue measure on $(\R^d,\|\,{\cdot}\,\|_p)$. This notion should not be confused with the usual notion of $p$-superharmonicity associated with the nonlinear $p$-Laplacian \cite{HKM2006}.

Recall that, for a smooth function $u\in C^2(\R^d)$, superharmonicity can also refer to the condition $\Delta u\leq0$. We will call this \textit{superharmonicity in the classical sense}. The following lemma connects this notion with mean superharmonicity when the given density is smooth. The reader may also consult \cite{Kijowski2020} for a study of functions satisfying the mean value equality over balls associated with $p$-norms.

\begin{lemma}\label{lem:p-implies-laplacian}
For every $1\leq p\leq\infty$, $p$-superharmonicity in the mean sense implies classical superharmonicity for functions in $C^2(\R^d)$, and the converse holds in the Euclidean case.
\end{lemma}

\begin{proof}
We will carry out the proof of the first implication for $d=2$. The same argument works in any dimension. Let $u\in C^2(\R^2)$ be $p$-superharmonic in the mean sense. Its Taylor expansion around $(0,0)$ is
\begin{align*}
    u(x,y) &= u(0,0) + \frac{\partial u}{\partial x}(0,0)x + \frac{\partial u}{\partial y}(0,0)y \\
    &\quad + \frac{1}{2}\frac{\partial^2 u}{\partial x^2}(0,0)x^2 + \frac{1}{2}\frac{\partial^2 u}{\partial y^2}(0,0)y^2 
    + \frac{\partial^2 u}{\partial x\partial y}(0,0)xy + o(\|(x,y)\|_2^2) \\
    &\eqdef u(0,0) + u_x x + u_y y + \frac{1}{2}u_{xx}x^2 + \frac{1}{2}u_{yy}y^2 + u_{xy}xy + o(\|(x,y)\|_2^2).
\end{align*}
Using the symmetries of the balls $B_p(0,\varepsilon)$, $(x,y)\in B_p(0,\varepsilon)$ if and only if $(\sigma_1 x,\sigma_2 y)\in B_p(0,\varepsilon)$ for every $\sigma_1,\sigma_2\in\{-1,1\}$ and $(y,x)\in B_p(0,\varepsilon)$, we get

\begin{align*}
    \fint_{B_p(0,\varepsilon)}u 
    &= u(0,0) + u_x \fint_{B_p(0,\varepsilon)}x + u_y \fint_{B_p(0,\varepsilon)}y + \frac{1}{2}u_{xx} \fint_{B_p(0,\varepsilon)}x^2 \\
    &\quad + \frac{1}{2}u_{yy} \fint_{B_p(0,\varepsilon)}y^2 + u_{xy} \fint_{B_p(0,\varepsilon)}xy + o(\varepsilon^2) \\
    &= u(0,0) + \frac{1}{2}u_{xx} \fint_{B_p(0,\varepsilon)}x^2 + \frac{1}{2}u_{yy} \fint_{B_p(0,\varepsilon)}y^2 + o(\varepsilon^2) \\
    &= u(0,0) + \frac{1}{2}\Delta u(0,0) \fint_{B_p(0,\varepsilon)}x^2 + o(\varepsilon^2),
\end{align*}
where all integrals are with respect to the Lebesgue measure.

By the mean inequality,

\[
    \fint_{B_p(0,\varepsilon)}u\,d\lambda
    \leq\fint_{B_p(0,\varepsilon/2^n)}u\,d\lambda
    \longrightarrow u(0,0).
\]
On the other hand,
\begin{align*}
    \fint_{B_p(0,\varepsilon)}x^2\,d\lambda
    \gtrsim\fint_{B_p(0,\varepsilon/2)}x^2\,d\lambda
    +\fint_{B_p(0,\varepsilon)\setminus B_p(0,\varepsilon/2)}x^2\,d\lambda
    \geq\fint_{B_p(0,\varepsilon)\setminus B_p(0,\varepsilon/2)}x^2\,d\lambda
    \gtrsim\varepsilon^2.
\end{align*}
Thus,
\[
    \Delta u(0,0)
    \lesssim\frac{|o(\varepsilon^2)|}{\displaystyle \fint_{B_p(0,\varepsilon)}x^2\,d\lambda}
    \lesssim \frac{|o(\varepsilon^2)|}{\varepsilon^2}
    \xrightarrow{\varepsilon\to0}0.
\]
We have only used the mean inequality and the symmetries of the balls, so the same argument shows that $\Delta u(x,y)\leq0$ for all $(x,y)\in\R^2$.

Conversely, if $\Delta u\leq0$, it is well known that the averages of $u$ over Euclidean spheres are decreasing in the radius. Averaging these spherical means over the corresponding Euclidean balls gives the same monotonicity for ball averages, so $u$ is $2$-superharmonic in the mean sense.
\end{proof}

Thanks to Liouville's theorem, we can now rule out non-constant smooth densities for doubling minimizers in the plane.

\begin{lemma}\label{prop:smooth-plane}
Let $1\leq p\leq\infty$. If $d\mu=u\,d\lambda$ is a doubling minimizer on $(\R^2,\|\,{\cdot}\,\|_p)$ with $u\in C^2(\R^2)$, then $u$ is constant and $\mu$ is a multiple of the Lebesgue measure.
\end{lemma}

\begin{proof}
The density $u$ satisfies the superharmonic mean value property, so the preceding lemma gives $\Delta u\leq0$. Since $u\geq0$, the function $-u$ is subharmonic and bounded above. Liouville's theorem \cite[Theorem 29]{ProtterWeinberger1984} implies that $u$ is constant.
\end{proof}

Therefore, this lemma leaves non-smooth densities as the only possibility for finding non-constant minimizers on $(\R^2,\|\,{\cdot}\,\|_p)$. To rule out such densities, we first observe that convolution preserves the minimizing condition.

\begin{prop}\label{prop:convolution}
Let \(1\leq p\leq\infty\), \(u\,d\lambda\in
\DM(\R^d,\|\,{\cdot}\,\|_p)\), and \(g\in L^1(\R^d)\) be nonnegative, nonzero, and compactly
supported. Then
\[(u*g)\,d\lambda\in\DM(\R^d,\|\,{\cdot}\,\|_p).\]
\end{prop}

\begin{proof}
Using Fubini and a change of variables, we get
\begin{align*}
    \fint_{B_p(x,2r)}(u*g)\,d\lambda
    &=
    \int_{\R^d}g(y)\left(\fint_{B_p(x-y,2r)}u(z)\,d\lambda(z)\right)\,d\lambda(y)  \\
    &\leq
    \int_{\R^d}g(y)\left(\fint_{B_p(x-y,r)}u(z)\,d\lambda(z)\right)\,d\lambda(y)
    =
    \fint_{B_p(x,r)}(u*g)\,d\lambda.\qedhere
\end{align*}
\end{proof}

We may now use a standard mollification argument to achieve uniqueness.

\begin{theorem}\label{teor:minimizadoresR2}
For every $1\leq p\leq\infty$,
\vspace{0.15cm}
\[
    \DM(\R^2,\|\,{\cdot}\,\|_p)=\{\alpha\lambda:\alpha>0\}.
\]
\end{theorem}

\begin{proof}
Let \(\mu\in\DM(\R^2,\|\,{\cdot}\,\|_p)\). By
Proposition \ref{minimizersabsolcont}, write
\(
    \mu=u\,d\lambda \text{ with } u\in L^1_{\mathrm{loc}}(\R^2).
\)
Let \(\rho\in C_c^\infty(\R^2)\) be nonnegative with integral \(1\), and set
\[
    \rho_\varepsilon(x)=\varepsilon^{-2}\rho(x/\varepsilon).
\]
By Proposition \ref{prop:convolution},
\((u*\rho_\varepsilon)d\lambda\) is a minimizer for every
\(\varepsilon>0\). Its density is smooth and nonnegative, so
Lemma \ref{prop:smooth-plane} gives
\[
    u*\rho_\varepsilon=c_\varepsilon
\]
for some constant \(c_\varepsilon>0\).

Since \(u*\rho_\varepsilon\to u\) in \(L^1_{\mathrm{loc}}(\R^2)\), the
constants \(c_\varepsilon\) converge as \(\varepsilon\downarrow0\) because they are Cauchy. Indeed,
on $B_2(0,1)$,
\[
    \lambda(B_2(0,1))|c_\varepsilon-c_\delta|
    \leq
    \int_{B_2(0,1)}|c_\varepsilon-u|\,d\lambda
    +
    \int_{B_2(0,1)}|u-c_\delta|\,d\lambda.
\]
Let \(c\geq0\) be the limit.
For every \(R>0\),
\[
    \int_{B_2(0,R)}|u-c|\,d\lambda
    \leq
    \int_{B_2(0,R)}|u-c_\varepsilon|\,d\lambda
    +
    \lambda(B_2(0,R))|c_\varepsilon-c|\to0.
\]
Thus \(u=c\) almost everywhere. Since \(\mu\) is nonzero, \(c>0\), and
\(\mu=c\lambda\).
\end{proof}

It remains to establish the higher-dimensional assertion in Theorem \ref{teor:Rd}. In contrast with the uniqueness obtained in lower dimensions, we exhibit an uncountable family of nonconstant minimizers on $(\R^d,\|\,{\cdot}\,\|_2)$ when $d\geq3$. By Lemma \ref{lem:p-implies-laplacian}, this is a direct consequence of the existence of positive, nonconstant functions that are superharmonic in the classical sense in these dimensions.

For instance, for $d\geq3$, the functions
\[
    u_\alpha(x)=(1+\|x\|_2^2)^{-\alpha},
    \qquad 0<\alpha\leq\frac{d-2}{2},
\]
are positive and smooth, and a direct computation shows that $\Delta u_\alpha\leq0$. Hence, $u_\alpha\,d\lambda$ is a minimizer on $(\R^d,\|\,{\cdot}\,\|_2)$. As these functions are linearly independent, they generate an infinite-dimensional cone.

Another family is given by the truncated Newtonian kernels
\[
    f_k(x)=
    \begin{cases}
        k^{2-d}, & \|x\|_2<k,\\
        \|x\|_2^{2-d}, & \|x\|_2\geq k,
    \end{cases}
    \qquad k>0.
\]
These functions satisfy the Euclidean superharmonic mean value property, and hence $f_k\,d\lambda$ is a minimizer on $(\R^d,\|\,{\cdot}\,\|_2)$.

For $p\neq2$, Lemma \ref{lem:p-implies-laplacian} gives only a necessary condition: a smooth density that is $p$-superharmonic in the mean sense must satisfy $\Delta u\leq0$. The Euclidean superharmonic mean value property need not hold for other norms, as the next example shows.

\begin{example}\label{ex:infty-counterexample}
The function
\[
    f_{1/2}(x)=
    \begin{cases}
        2, & \|x\|_2<1/2,\\
        \|x\|_2^{-1}, & \|x\|_2\geq1/2,
    \end{cases}
    \qquad x\in\R^3,
\]
is Euclidean superharmonic, but not $\infty$-superharmonic, in the mean sense. Indeed, the Euclidean statement follows from the family $f_k$ above with $d=3$ and $k=1/2$.
Let $h(x)=\|x\|_2^{-1}$ for $x\neq0$ and $e=(1,1,0)/\sqrt2$. For sufficiently small $s>0$, define
\[
    A(s)=\fint_{[-s,s]^3}h(e+z)\,d\lambda(z).
\]
By symmetry, monomials with an odd exponent in any coordinate have zero average over $[-s,s]^3$. Moreover,
\[
    \fint_{[-s,s]^3}z_i^2\,d\lambda=\frac{s^2}{3},\qquad
    \fint_{[-s,s]^3}z_i^4\,d\lambda=\frac{s^4}{5},\qquad
    \fint_{[-s,s]^3}z_i^2z_j^2\,d\lambda=\frac{s^4}{9}\quad(i\neq j).
\]
Since $h$ is harmonic near $e$, the quadratic term $s^2\Delta h(e)/6$ vanishes. Expanding $h$ at $e$ to fifth order and averaging therefore gives
\[
    A(s)=1+s^4\left(\frac1{120}\sum_i\partial_i^4h(e)
              +\frac1{36}\sum_{i<j}\partial_i^2\partial_j^2h(e)\right)+O(s^6).
\]
The identity $\Delta^2h(e)=0$ also gives
\[
    \sum_{i<j}\partial_i^2\partial_j^2h(e)
    =-\frac12\sum_i\partial_i^4h(e).
\]
Consequently,
\[
    A(s)=1-\frac{s^4}{180}\sum_i\partial_i^4h(e)+O(s^6)
         =1+\frac7{120}s^4+O(s^6),
\]
where we used $\partial_i^4h(e)=105e_i^4-90e_i^2+9$ and hence $\sum_i\partial_i^4h(e)=-21/2$. It follows that
\[
    A(2s)-A(s)=\frac78s^4+O(s^6)>0
\]
for all sufficiently small $s>0$.

Choose such an $s$ with $2\sqrt3s<1/2$. If $x\in B_\infty(e,2s)$, then
\[
    \|x\|_2\geq\|e\|_2-\|x-e\|_2>1-2\sqrt3s>1/2.
\]
\enlargethispage{\baselineskip}
Hence $f_{1/2}=h$ on both $B_\infty(e,s)$ and $B_\infty(e,2s)$, and
\[
    \fint_{B_\infty(e,2s)}f_{1/2}\,d\lambda
    =A(2s)>A(s)
    =\fint_{B_\infty(e,s)}f_{1/2}\,d\lambda.
\]
This shows that $f_{1/2}$ is not $\infty$-superharmonic in the mean sense.
\end{example}

The above discussion shows that the Euclidean superharmonic mean value property does not, in general, produce minimizers for non-Euclidean norms.
This obstruction is consistent with the rigidity of the mean-value property studied in \cite{EpsteinSchiffer1965, GoldsteinOw1971}. In particular, for a bounded convex domain $D\subseteq\R^d$, $d\geq2$, and a point $t\in D$, the identity
\[
    \fint_D h\,d\lambda=h(t)
\]
for any integrable harmonic function $h$ on $D$ forces $D$ to be a Euclidean ball centered at $t$.

The result relates directly to why classical superharmonicity alone is not enough to ensure the corresponding mean inequalities for non-Euclidean $\|\,{\cdot}\,\|_p$\kern0.1em-balls. Suppose that $B=B_p(t,R)$ has the following property: for every integrable function $u\in C^2(B)$ with $\Delta u\leq0$,
\[
    \fint_{B_p(t,r)}u\,d\lambda
    \leq
    \fint_{B_p(t,s)}u\,d\lambda,
    \qquad \text{for every }0<s<r\leq R.
\]
Then $B$ must be a Euclidean ball. Indeed, if $h$ is an integrable harmonic function on $B$, both $h$ and $-h$ satisfy the hypothesis. Applying the inequality to both functions gives
\[
    \fint_B h\,d\lambda
    =
    \fint_{B_p(t,s)}h\,d\lambda,
    \qquad 0<s<R.
\]
Letting $s\to0$,
\[
    \fint_B h\,d\lambda=h(t).
\]
The result above therefore forces $B$ to be a Euclidean ball centered at $t$. Thus one should not expect $\|\,{\cdot}\,\|_p$\kern0.1em-balls with $p\neq2$ to interact with the Laplacian in the same way as Euclidean balls.

\begin{remark}
As established above, the Lebesgue measure is the only doubling minimizer, up to multiplication by a constant, in dimensions $1$ and $2$, while the Euclidean norm admits an abundance of nonconstant minimizers in higher dimensions. It is natural to ask what happens for the other $p$-norms when $d\geq3$. Our Euclidean examples rely on the connection with classical superharmonicity, but the preceding discussion shows why this construction does not extend directly to other norms. This leaves open both the possibility of finding nonconstant minimizers and that of the Lebesgue measure being the unique minimizer up to multiplication by a constant. Whether either behaviour holds for some or all $p\neq2$ is, to the best of our knowledge, open at the moment.
\end{remark}

\section*{Statement on AI use}
A few months after a first version of this paper was finished, LLMs versions \texttt{gpt5.5} and \texttt{gemini3.0} were used to improve some of the results. In particular, these tools assisted in building Example \ref{ex:noncomplete-counterexample}, amending parts of the proof of Theorem \ref{t:existence}, improving the approximation argument in the proof of Proposition \ref{prop:range}, and using mollifiers in Theorem \ref{teor:minimizadoresR2}. Codex was additionally used to assist with the exposition and revision of the manuscript. All mathematical arguments developed with AI assistance were checked in detail by the authors. The authors take full responsibility for the mathematical content and final presentation of the paper.

\section*{Acknowledgments}
We are grateful to Eugenia Malinnikova for interesting discussions and in particular for sharing the argument leading to the proof of Theorem \ref{teor:graphs}. We are also thankful to Przemys\l aw G\'orka and Franciszek Gajownik for spotting a gap in an earlier version of Theorem \ref{t:existence}.

\bibliographystyle{alpha}
\bibliography{newVersion_merged}

\end{document}